\documentclass[a4paper,11pt,fleqn]{article}
\usepackage{amsmath,amssymb,amsthm,graphicx,subfigure,float,caption,epstopdf,tabularx,color, bm,amsfonts,epic}
\usepackage[top=1.25in, bottom=1.2in, left=0.8in, right=0.8in]{geometry}
\usepackage{appendix}
\usepackage{multirow}
\usepackage{xcolor}
\usepackage{setspace}
\usepackage{booktabs}
\usepackage{appendix}
\usepackage{multirow}
\usepackage{xcolor}
\usepackage{setspace}
\usepackage{diagbox}
\usepackage{booktabs}
\usepackage[linkcolor=red,anchorcolor=red,citecolor=blue,colorlinks]{hyperref}
\usepackage{bbm}
\newtheorem{thm}{Theorem}[section]

\newtheorem{rmk}{Remark}[section]
\newtheorem{sch}{Scheme}[section]

\numberwithin{equation}{section}

\usepackage{indentfirst}
\graphicspath{{Fig}}

\addcontentsline{toc}{section}{References}

\begin{document}
\title {Arbitrarily high-order energy-preserving schemes for the Zakharov-Rubenchik equations}

\author{ Gengen Zhang$^{1}$, Chaolong Jiang$^{2,3}$\footnote{Correspondence author. Email: Chaolong$\_$jiang@126.com}, Hao Huang$^{2,4}$\\
{\small$^1$ School of Mathematics and Statistics, Yunnan University, Kunming 650504,  China}
\\ 
{\small$^2$ School of Statistics and Mathematics,
Yunnan University of Finance and Economics, Kunming 650221, China}\\
{\small $^3$ Department of Mathematics, College of Liberal Arts and Science, National University of Defense Technology,}\\
{\small  Changsha 410073, China}\\
{\small $^4$ Basic Teaching Department, Guizhou Vocational Institute of Technology, Fuquan 550500, China} }

\date{}

\maketitle

\begin{abstract}
In this paper, we present a novel class of high-order energy-preserving schemes for solving the Zakharov-Rubenchik equations. The main idea of the scheme is first to introduce an quadratic auxiliary variable to transform the Hamiltonian energy into a modified quadratic energy and the original system is then reformulated into an equivalent system which satisfies the mass, modified energy as well as two linear invariants. The symplectic Runge-Kutta method in time, together with the Fourier pseudo-spectral method in space is employed to compute the solution of the reformulated system. The main benefit of the proposed schemes is that it can achieve arbitrarily high-order accurate in time and conserve the three invariants: mass,  Hamiltonian energy and two linear invariants. In addition, an efficient fixed-point iteration is proposed to solve the resulting nonlinear equations of the proposed schemes. Several experiments are addressed to validate the theoretical results.
\end{abstract}


\section{Introduction}

The Zakhrov-Rubenchik  (ZR) equations are applied extensively to describe the interaction and the dynamics of spectrally narrow high-frequency waves with low-frequency (acoustic) waves \cite{ZR72}.
In this paper, we consider the following ZR equations
\begin{equation}\label{ZR1}
\left\{
\begin{split}
& {\rm i} \partial_tB(x,t)  + \omega  \partial_{xx}B(x,t) - \kappa \big( u(x,t) - \frac{1}{2}\nu \rho(x,t) + q|B(x,t)|^2 \big) B(x,t)   =0  , \\
& \partial_t\rho(x,t)  +  \partial_x\big( u(x,t) - \nu \rho(x,t) \big) = - \kappa   \partial_x(|B(x,t)|^2)  ,\\
& \partial_tu(x,t)  +  \partial_x\big( \beta \rho(x,t) - \nu u(x,t) \big) = \frac{1}{2}\kappa \nu   \partial_x(|B(x,t)|^2), \ {x} \in \mathbb{R},\ t>0,
 \end{split}\right.
\end{equation}
with the initial conditions
\begin{align}
B({x}, 0) = B_0({x}), \   \rho({x}, 0) = \rho_0({x}),
\   u({x}, 0) = u_0({x}), \  {x} \in \mathbb{R},  \label{ZRcon}
\end{align}
where the complex-valued function $B:=B({x}, t)$ denotes the
magnetic field generated by the plasma, $\rho:=\rho({x}, t)$ and $u:=u({x},t)$ are the real-valued functions, which
represent the density of mass and the fluid speed, respectively. Moreover, $\omega$, $\kappa$, $\nu$, $\beta$ and $q$ are given real
constants where $\omega$ and $\kappa$ denote the frequency and the wave number, respectively, $\nu$ and $\beta$ are dimensionless parameters, and $q$ satisfies
\begin{align*}
q = \kappa + \frac{\nu(\kappa \nu -1 )}{4(\beta-\nu^2)}.
\end{align*}

The ZR equations \eqref{ZR1} conserve the following three invariants:
\begin{itemize}
\item Mass
\begin{align}\label{CONmass}
& \mathcal{M}(t)= \int_{\mathbb{R}} |B|^2 d{x}\equiv  \mathcal{M}(0),\ t\ge0.
\end{align}
\item Hamiltonian energy
\begin{align}\label{CONenergy}
\mathcal{H}(t) = \int_{\mathbb{R}} \Big( - \omega |B_x |^2 - \kappa \big( u - \frac{\nu}{2} \rho
      +\frac{q}{2}|B|^2 \big) | B|^2
    -\frac{\beta}{2} \rho^{2} -\frac{1}{2} u^{2} + \nu u \rho  \Big) d {x},\ t\ge 0.
    \end{align}
\item Two linear invariants
\begin{align}\label{CONI1}
& \mathcal{I}_1(t) = \int_{\mathbb{R}} \rho~ d{x}\equiv  \mathcal{I}_1(0),\  \mathcal{I}_2(t) = \int_{\mathbb{R}} ud{x}\equiv  \mathcal{I}_2(0),\  t\ge0.
\end{align}
\end{itemize}

Extensive theoretical studies have been carried out for the equations \eqref{ZR1} in the literature.
For the well-posedness and existence of the ZR equations \eqref{ZR1}, we refer to \cite{Linares09,Oliveira03pd,Oliveira15,Ponce05} and references therein.
In particular, Oliveira \cite{Oliveira03pd} not only proved the global well-posedness of the ZR equations \eqref{ZR1}
in the space $H^2(\mathbb{R})\times H^1(\mathbb{R})\times H^1(\mathbb{R})$, but also proved the existence and the orbital stability of the solitary wave solutions.
For the asymptotic behavior of solutions and the adiabatic limit of the ZR equations \eqref{ZR1}, we refer to \cite{Cordero16} and \cite{Oliveira08}, respectively for more details.

A numerical scheme that preserves one or more invariants of the original system is known as an energy-preserving scheme. It is shown in \cite{SV1986IMA,ZFPV95} that the non-energy-preserving scheme will produce nonlinear blow-up in the numerical simulation of soliton collisions. Thus, over the last decade, there are some literatures concerning the energy-preserving schemes for the ZR equations \eqref{ZR1}.
Zhao et al. \cite{Zhao14jsc} proposed a time-splitting Fourier pseudo-spectral (TS-FP) scheme and a
Crank-Nicolson finite difference (CN-FD) scheme, respectively. It is proven that the TS-FP scheme can conserve the mass \eqref{CONmass}, and the CN-FD scheme can
conserve three invariants \eqref{CONmass}-\eqref{CONI1}. Later on, the Crank-Nicolson Fourier pseudo-spectral (CN-FP) scheme is investigated and analysed in \cite{WangZhang19acm}. Meanwhile, Ji et al. \cite{JiZhang19ijcm} proposed a Crank-Nicolson compact finite difference scheme, which is proven rigorously in mathematics that the scheme is second-order accurate in time and fourth-order accurate in space, respectively. Nevertheless, all of existing energy-preserving schemes are second-order accurate in time. In \cite{GZW2020jcp,JCQSjsc2022}, numerical experiments show that the high-order energy-preserving schemes not only provide much smaller numerical error but also more robust than the second-order accurate schemes. Therefore, it is interesting to design high-order energy-preserving schemes for solving the ZR equations \eqref{ZR1}, which motives this paper.

In 1987, Cooper\cite{Cooper1987} showed that all RK methods conserve linear invariants and an irreducible RK
method can preserve all quadratic invariants if and only if their coefficients satisfies the specific algebraic condition. Then, Calvo et al. \cite{CIZ1997-MP} proved that no RK method can preserve arbitrary polynomial invariants of degree 3 or higher of arbitrary vector fields. Consequently, over the past few decades, there has been an increasing interest in higher-order energy-preserving methods for general conservative systems. The noticeable ones include high-order averaged vector field methods \cite{LWQ14,MHW2021jcp,QM08}, Hamiltonian Boundary Value Methods (HBVMs) \cite{BI16,BIT10}, energy-preserving continuous stage Runge-Kutta methods \cite{MB16,TS12,WJJSC2022}, functionally
fitted energy-preserving methods \cite{LWsina16,MB16} and energy-preserving variant of collocation methods \cite{CH11bit,H10} etc. However, it is worth noting that these mentioned methods can  conserve both the Hamiltonian energy and the linear invariant of the original system at most (see \cite{BBCIamc18,LW15}). Thus, how to design high-order numerical schemes which can conserve the three invariants \eqref{CONmass}-\eqref{CONI1} of the ZR equations \eqref{ZR1} is challenging.

Recently, inspire by the energy quadratization (EQ) approach \cite{SXY18,SXY19siamrev,YZW17}, Gong et al\cite{GongQuezheng21,GHWW2022} first present the idea of the quadratic auxiliary variable (QAV) approach for developing high-order energy-preserving scheme for polynomial Hamiltonian partial differential equations, which was revisited and generalized  more recently by Tapley \cite{Tapley-SISC2022}. The basic idea of the QAV approach and EQ approach is first to transform the original Hamiltonian energy into a modified quadratic energy by introducing appropriate auxiliary variables, and based on the energy variational principle, a reformulated system, which satisfies the quadratic energy, is then obtained. Finally, various energy-preserving schemes are obtained using the quadratic invariant conserving method to the reformulated system. The diffidence between them lies in that, on the one hand, the auxiliary variable introduced by the QAV approach shall be quadratic, and it usually may introduce multiple quadratic auxiliary variables for the polynomial Hamiltonian energy of high degree \cite{JQSZ2022-arXiv,Tapley-SISC2022}. On the other hand, the developed high-order scheme based on the QAV approach can conserve the original energy of the general polynomial Hamiltonian system in stead of a modified one \cite{CWJ2021cpc,JWG19}. In this paper, based on the idea of the QAV approach, we propose a novel class of high-order energy-preserving scheme, which can conserve the three invariants \eqref{CONmass}-\eqref{CONI1} of the ZR equations \eqref{ZR1} exactly. This is known to be the first temporal high-order multiple invariants conserving scheme for the ZR equations \eqref{ZR1} in the literature and more interesting is that the classical CN-FP scheme \cite{WangZhang19acm} is a special case of our proposed methods.

The rest of this paper is organized as follows.
In section \ref{Sec-ZR-2}, based on the idea of the QAV approach, we first reformulate the original ZR equations \eqref{ZR1} into an equivalent system. In section \ref{Sec-ZR-3}, we present a class of high-order numerical schemes which can conserve the three invariants \eqref{CONmass}-\eqref{CONI1} of the ZR equations \eqref{ZR1}. An efficient implementation for solving the nonlinear equations of the proposed scheme is presented in section \ref{Sec-ZR-4}.
In section \ref{Sec-ZR-5}, extensive numerical tests and comparisons are carried out to illustrate
the performance of the proposed schemes. Section \ref{Sec-ZR-6} includes the conclusions of this paper.

\section{Model reformulation}\label{Sec-ZR-2}
Based on the idea of the QAV approach, the ZR equations \eqref{ZR1} are first reformulated into an equivalent form, which provides an elegant platform for efficiently
developing arbitrarily high-order energy-preserving schemes. Due to the rapid decay of the solution at the far field \cite{Linares09,Oliveira03pd},
the ZR equations \eqref{ZR1} are truncated on a bounded interval $\Omega=[a,b]$ with a periodic boundary condition. In addition, for any $f, g\in L^2(\Omega)$, we define the continuous inner product as $(f,g) = \int_{\Omega} f \bar{g} \mathrm{d} {x}$, in which $\bar{g}$ represents the conjugate of $g$, and the $L^2$-norm is denoted as $\|f\|^2 = (f,f)$.

To be the start, the ZR equations \eqref{ZR1} can be written as
\begin{align}
\begin{split}\label{ZR-Hamiltonian}
\frac{\partial z}{\partial t}
= \mathcal{D}~ \frac{\delta \mathcal{H}}{\delta \bar{z}},
  \end{split}\ \mathcal{D}=  \left(\begin{array}{ccc}
{\rm i}  &  0 & 0 \\
0 & 0 & \partial_x \\
0 & \partial_x & 0 \\
\end{array} \right),
 \end{align}
where $z=(B,\rho,u)^T $ and $ \frac{\delta \mathcal{H}}{\delta \bar{z}} $ is the variational derivative of the Hamiltonian functional \eqref{CONenergy}.
Then motivated by \cite{GongQuezheng21,JWG19}, we introduce an quadratic auxiliary variable as
\begin{equation}\label{auxiliary-variable}
	\phi:=\phi(x,t)=|B|^2.
\end{equation}
Then, the Hamiltonian energy functional \eqref{CONenergy} is reformulated into the following quadratic form
\begin{equation}\label{modified-energy}
	\mathcal{E}(t) =  \int_{\Omega} \Big( -\omega |B_x |^2 - \kappa \big( u - \frac{\nu}{2} \rho
      +\frac{q}{2}\phi \big) \phi
    -\frac{\beta}{2} \rho^{2} -\frac{1}{2} u^{2} +\nu u \rho  \Big) d {x}.
\end{equation}
 According to the energy variational principle, the system \eqref{ZR-Hamiltonian} can be transformed into an equivalent formulation, as follows:
\begin{equation}\label{QAVsystem}
	\begin{cases}
 & B_{t} = {\rm i} \big( \omega B_{xx} - \kappa \big( u - \frac{1}{2}\nu \rho + q \phi \big) B \big), \\
&  \rho_{t}  =    \partial_{x} \big( -u+ \nu \rho - \kappa  \phi \big)  ,\\
&  u_{t}  =   \partial_{x} \big( - \beta \rho + \nu u  + \frac{1}{2}\kappa \nu \phi \big),\\
&  \phi_t = 2 {\rm Re} (\bar{B}\cdot B_t),\ {x}\in \Omega,\ t\ge 0,
	\end{cases}
\end{equation}
where ${\rm Re}(\bullet)$ denotes the real part of $\bullet$,
and the consistent initial conditions are given by
\begin{equation}\label{CIC}
\begin{array}{lll}
B ({x},0) = B_0({x}), \   \rho ({x},0) = \rho_0({x}),
\  u({x},0) = u_0({x}), \
\phi({x},0) = \big| B_0({x})\big|^2,\ {x}\in \Omega.
\end{array}
\end{equation}

\begin{thm}\label{ZR-thm-2.1} With the periodic boundary condition, the system \eqref{QAVsystem}-\eqref{CIC} conserves the invariants
	\begin{align}\label{QAVCLphi}
		&H_1(x,t)=H_1(x,0)=0,\ H_1(x,t)=\phi({x},t)- |B({x},t)|^{2},\ {x}\in \Omega,\; t\ge 0
	\end{align}
and the modified quadratic energy \eqref{modified-energy}.
\end{thm}

\begin{proof}
It follows from the last equality of \eqref{QAVsystem} that
\begin{equation*}
\partial_tH_1(x,t)=\partial_t(\phi({x},t)-|B({x},t)|^2)=\partial_t\phi({x},t)-2{\rm Re}\Big(\bar{B}({x},t)\cdot\partial_tB({x},t)\Big)=0.
\end{equation*}
Thus, with the initial condition \eqref{CIC}, the equation \eqref{QAVCLphi} holds for any ${x}\in \Omega$ and $t\ge0$.

Subsequently, with the periodic boundary condition, we can deduce that
	\begin{align*}
		\frac{\mathrm{d}}{\mathrm{d}t}\mathcal{E}(t)
&=2{\rm Re}(\omega B_{xx},B_t)-\kappa\big(u-\frac{1}{2}\nu\rho+q\phi,\phi_t\big)+\big(-\beta\rho+\nu u+\frac{1}{2}\kappa\nu\phi,\rho_t\big)+\big(-u+\nu\rho-\kappa\phi,u_t\big).
\end{align*}
Using \eqref{QAVsystem} and the periodic boundary condition, we can obtain from the above equation
\begin{align*}
\frac{\mathrm{d}}{\mathrm{d}t}\mathcal{E}(t)&=2{\rm Re}\big(\omega B_{xx}-\kappa\big(u-\frac{1}{2}\nu\rho+q\phi\big)B,B_t\big)
         +\big(-\beta\rho+\nu u+\frac{1}{2}\kappa\nu\phi,\partial_{x}\big( -u+ \nu \rho-\kappa\phi\big)\big)\\
         &~~~~+\big(-u+\nu\rho-\kappa\phi,\partial_{x}\big(-\beta\rho+\nu u+\frac{1}{2}\kappa\nu\phi\big)\big) \\
		&=-2{\rm Re}\big({\rm i}B_t,B_t\big)+\big( - \beta \rho+\nu u+\frac{1}{2}\kappa\nu\phi,\partial_{x}\big(-u+\nu\rho-\kappa\phi\big)\big)\\
&~~~~-\big(-\beta\rho+\nu u+\frac{1}{2}\kappa\nu\phi,\partial_{x}\big(-u+\nu\rho-\kappa\phi\big)\big)\\
     &= 0.
\end{align*}
This completes the proof.
\end{proof}

\begin{thm}\label{ZR-thm-2.2} Under the periodic boundary condition, the system \eqref{QAVsystem}-\eqref{CIC} also conserves the mass \eqref{CONmass} and the two linear invariants \eqref{CONI1}.
\end{thm}
\begin{proof} We make the inner product of the first equation of \eqref{QAVsystem} with $2B$ and take the real
part of the resulting equation to obtain
\begin{align*}
\frac{d}{dt}\int_{\Omega}|B|^2dx=0,
\end{align*}
which implies $\frac{d}{dt}\mathcal{M}(t)=0,\ t\ge0$.

With noting the periodic boundary condition and the second and third equations of \eqref{QAVsystem}, we have, respectively,
\begin{align*}
& 	\frac{\mathrm{d}}{\mathrm{d}t}\int_{\Omega}\rho dx
= \int_{\Omega}\partial_{x}\big(-u+\nu\rho-\kappa\phi\big) d{x}=0,
\end{align*}
and
\begin{align*}
\frac{\mathrm{d}}{\mathrm{d}t}\int_{\Omega}u dx
= \int_{\Omega}\partial_{x}\big(-\beta\rho+\nu u+\frac{1}{2}\kappa\nu\phi\big)d{x}=0,
\end{align*}
which implies that the system \eqref{QAVsystem} conserves the two linear invariants \eqref{CONI1}. This completes the proof.
\end{proof}
\section{High-order energy-preserving scheme}\label{Sec-ZR-3}

In this section, we will propose a novel class of high-order energy-preserving schemes,
which are based on the Fourier pseudo-spectral method in space and the symplectic RK method in time for the reformulated system \eqref{QAVsystem}, respectively. We first apply the Fourier pseudo-spectral method to the system \eqref{QAVsystem} in space and a conserving semi-discrete system is given in Section \ref{ZR-FPM} and the fully discrete scheme is then proposed in Section \ref{ZR-fully-scheme}.

\subsection{Spatial semi-discretization}\label{ZR-FPM}
Let $\Omega_{h}=\{x_{j}|x_{j}=a+jh,\ 0\leq j\leq N\}$ be a partition of $\Omega$ with spatial mesh size $h= (b-a)/N$, where $N$ is a positive even number. Let $U_j$ be the numerical approximation of $U(x_j,t)$ for $j=0,1,\cdots,N$; denote $\mathbb{V}_{h}=\big\{{\bm U}|{\bm U}=(U_{0},U_{1},\cdots,U_{N-1})^{T}\big\}$ be the space of mesh functions on $\Omega_h$ that satisfies the periodic boundary condition.
Define the discrete inner product and $L^2$- and $L^{\infty}$-norms, respectively, as
\begin{align*}
&\langle {\bm U},{\bm V}\rangle_{h}=h\sum_{j=0}^{N-1}U_{j}\bar{V}_{j},\
\|\bm U \|_{h} =  \langle {\bm U},{\bm U}\rangle_{h}^{\frac{1}{2}},  \
\|{\bm U}\|_{h,\infty}=\max\limits_{0\le j\le N-1}|U_{j}|, \  \forall\  {\bm U},{\bm V}\in\mathbb{ V}_{h}.
\end{align*}
We also denote $`\cdot$' as the componentwise product of vectors ${\bm U},{\bm V}\in\mathbb{V}_h$, i.e.,
\begin{align*}
{\bm U}\cdot {\bm V}=&\big(U_{0}V_{0},U_{1}V_{1},\cdots,U_{N-1}V_{N-1}\big)^{T}.
\end{align*}
For brevity, we denote ${U}\cdot U$ and ${U}\cdot \bar{U}$ as $U^2$ and $|U|^2$, respectively.

Let $g_{j}(x)$ be the interpolation basis function given by
\begin{align*}
  &g_{j}(x)=\frac{1}{N}\sum_{l=-N/2}^{N/2}\frac{1}{c_{l}}e^{\text{i}l\mu (x-x_{j})},\ c_{l}=\left \{
 \aligned
 &1,\ |l|<\frac{N}{2},\\
 &2,\ |l|=\frac{N}{2},
 \endaligned
 \right.
\end{align*}
where $\mu=\frac{2\pi}{b-a}$.
Then, we denote
\begin{align*}
S_{N}=\text{span}\{g_{j}(x),\ 0\leq j\leq N-1\}
\end{align*}
as the interpolation space, and the interpolation operator $I_{N}: C(\Omega)\to S_{N}$ is then defined as \cite{CQ01}
\begin{align*}
I_{N}U(x,t)=\sum_{j=0}^{N-1}U_{j}(t)g_{j}(x),
\end{align*}
where $U_{j}(t)=U(x_{j},t)$. Taking the partial derivative with respect to $x$, and we then evaluate the resulting expression at the collocation point $x_{j}$ to obtain
\begin{align*}
\frac{\partial^{m} I_{N}U(x_{j})}{\partial x^{m}}
&=\sum_{k=0}^{N-1}U_{k}\frac{d^{m}g_{k}(x_{j})}{dx^{m}}=({\mathcal D}_{m}{\bm U})_{j},\ {\bm U}\in\mathbb{V}_h,
\end{align*}
where $j=0,\cdots,N-1$ and ${\mathcal D}_{m}$ is an $N\times N$ matrix with elements given by
\begin{align*}
({\mathcal D}_{m})_{j,k}=\frac{d^{m}g_{k}(x_{j})}{dx^{m}},\ j,k=0,1,\cdots,N-1.
\end{align*}
Especially, it holds for ${\mathcal D}_{1}$ and ${\mathcal D}_{2}$ that
\begin{align}
&{\mathcal D}_{1}=\mathcal{ F}_{N}^{H}\Lambda^1\mathcal{F}_{N},\ \Lambda^1=\text{i}\mu\cdot\text{diag}\big(0,1,\cdots,\frac{N}{2}-1,0,-\frac{N}{2}+1,\cdots,-1\big), \label{SP-D1}\\
&{\mathcal D}_{2}=\mathcal{F}_{N}^{H}\Lambda^2\mathcal{F}_{N},\ \Lambda^2=\big[\text{i}\mu\cdot\text{diag}(0,1,\cdots,\frac{N}{2}-1,\frac{N}{2},-\frac{N}{2}+1,\cdots,-1)\big]^2, \label{SP-D2}
\end{align}
 where  $\mathcal{F}_{N}$ is the discrete Fourier transform matrix with elements
$\big(\mathcal{F}_{N}\big)_{j,k}=\frac{1}{\sqrt{N}}e^{-\text{\rm i}jk\frac{2\pi}{N}},$ $\mathcal{F}_{N}^{H}$ is the conjugate transpose matrix of $\mathcal{F}_{N}$. For more details, please refer to \cite{GCW14CiCP,ST06}.

Applying the Fourier pseudo-spectral method to approximate the system \eqref{QAVsystem} in space yields
\begin{align}\label{QAVsystem-SP}
\left\{
\begin{aligned}
&\frac{d}{dt} {\bm B }={\rm i}\big(\omega {\mathcal D}_{2}{\bm B}-\kappa\big({\bm u}-\frac{1}{2}\nu {\bm \rho}+q{\bm\phi}\big)\cdot{\bm B}\big), \\
&\frac{d}{dt} {\bm \rho}= {\mathcal D}_{1}\big(-{\bm u}+\nu{\bm \rho}-\kappa{\bm \phi}\big),\\
&\frac{d}{dt} {\bm u}={\mathcal D}_{1}\big(-\beta {\bm \rho}+\nu{\bm u}+\frac{1}{2}\kappa\nu{\bm\phi}\big),\\
&\frac{d}{dt} {\bm \phi}=2{\rm Re}({\bm B}_t\cdot\bar{{\bm B}}),
	\end{aligned}
\right.
\end{align}
where ${\bm B}, {\bm u}, {\bm \rho}, {\bm \phi}\in\mathbb{V}_h$. Then, we prove that the semi-discrete system \eqref{QAVsystem-SP} conserves the semi-discrete invariants of Theorems \ref{ZR-thm-2.1} and \ref{ZR-thm-2.2}.

   \begin{thm} The semi-discrete system \eqref{QAVsystem-SP} conserves the following semi-discrete invariants
  \begin{itemize}
\item Mass
  \begin{align}
  \mathcal M_h(t) = \mathcal M_h(0),\ \mathcal{M}_h=\langle|{\bm B}|^2,{\bm 1}\rangle_h,\ {\bm B} \in\mathbb{V}_h,\ t\ge0,
  \end{align}
  where ${\bm 1}=[1,1,\cdots,1]^T\in\mathbb{R}^{N}.$
  \item Two quadratic invariants
   \begin{align}
\ {\bm H}_1(t)={\bm H}_1(0),\ \mathcal{E}_h(t) = \mathcal{E}_h(0),\ {\bm H}_1(t) \in\mathbb{V}_h,\ t\ge0,
  \end{align}
  where
  \begin{align*}
  &{\bm H}_1(t)=\bm{\phi}(t)-|{\bm B}(t)|^{2}, \\
  & \mathcal{E}_h(t) = \omega \langle{\mathcal D}_{2} {\bm B} , {\bm B} \rangle_h - \kappa \langle{\bm u} - \frac{\nu}{2} {\bm \rho}
      + \frac{q}{2} {\bm \phi} , {\bm \phi} \rangle_h - \frac{\beta}{2} \langle{\bm \rho} ,{\bm \rho}\rangle_h
      -\frac{1}{2}\langle{\bm u}, {\bm u}\rangle_h +\nu\langle{\bm u}, {\bm \rho}\rangle_h.
      \end{align*}
  \item Two linear invariants
  \begin{align}
  &\mathcal I_{1h}(t) = \mathcal I_{1h}(0), \ \mathcal I_{2h}(t) = \mathcal I_{2h}(0),\ t\ge0,
  \end{align}
  where
  \begin{align*}
  \mathcal{I}_{1h}=\langle{\bm \rho} ,{\bm 1} \rangle_h,\ \mathcal{I}_{2h}=\langle{\bm u} ,{\bm 1} \rangle_h,\ {\bm \rho},{\bm u}\in\mathbb{V}_h.
  \end{align*}
\end{itemize}
   \end{thm}
   \begin{proof}
 The proof is similar to that of Theorems \ref{ZR-thm-2.1} and \ref{ZR-thm-2.2}, thus we omit it here for brevity.
   \end{proof}

\subsection{Full discretization}\label{ZR-fully-scheme}

Let $t_n = n \tau$ and  $t_{ni} = n \tau +c_i \tau$,  $n=0,1,2,\cdots $, $i=1,2,\cdots, s$,
where $\tau$ is the time step. Let $W_j^n$ and $(W_{ni})_j$ be the numerical approximation to the function $W(x,t)$ at points  $(x_j,t_n)$ and $(x_j,t_{ni})$ where $j=0,1,2,\cdots,N-1$ and $n=0,1,2,\cdots$, respectively.
Using an $s$-stage RK method to discrete the semi-discrete  system \eqref{QAVsystem-SP} in time,
we then obtain the following fully discrete scheme.

\begin{sch}\label{scheme:QAV-RK}
Let $b_i,a_{ij}({i,j=1,\cdots,s})$ be real numbers and let $c_i=\sum_{j=1}^sa_{ij}$.
 For given ${\bm B}^{n},{\bm \rho}^{n}, {\bm u}^{n},{\bm \phi}^n\in\mathbb{V}_h$, an $s$-stage RK method is given by
	\begin{equation}\label{TDIV}
		\begin{cases}
			{\bm B}_{ni} =  {\bm B}^n +  \tau \sum\limits_{j=1}^s a_{ij} k_j^{1}, \ k_i^{1} ={\rm i} \big( \omega {\mathcal D}_{2} {\bm B}_{ni}  - \kappa \big( {\bm u}_{ni} - \frac{1}{2}\nu{\bm \rho}_{ni}
                  + q {\bm \phi}_{ni} \big)\cdot  {\bm B}_{ni} \big), \\
	    	{\bm \rho}_{ni} =  {\bm \rho}^n +  \tau \sum\limits_{j=1}^s a_{ij} k_j^{2}, \ k_i^{2} =  {\mathcal D}_{1} \big( -{\bm u}_{ni} + \nu {\bm \rho}_{ni}  - \kappa {\bm \phi}_{ni} \big), \\
			{\bm u}_{ni} = {\bm u}^n +  \tau \sum\limits_{j=1}^s a_{ij}  k_j^{3}, \ k_i^{3}=  {\mathcal D}_{1}  \big( - \beta{\bm \rho}_{ni} + \nu {\bm u}_{ni}  + \frac{1}{2}\kappa \nu {\bm \phi}_{ni} \big), \\
			{\bm \phi}_{ni} ={\bm \phi}^n +  \tau \sum\limits_{j=1}^s a_{ij}  k_j^{4}, \ k_i^{4}=   2 {\rm Re} ( \bar{{\bm B}}_{ni} \cdot k_i^{1} ),\ i=1,2,\cdots,s,
		\end{cases}
	\end{equation}
	and ${\bm B}^{n+1},{\bm \rho}^{n+1}, {\bm u}^{n+1}$ and ${\bm \phi}^{n+1}$ are then updated by
	\begin{align}
		{\bm B}^{n+1} =  {\bm B}^n +  \tau \sum\limits_{i=1}^s b_{i} k_i^{1},\label{TD-E} \\
        {\bm \rho}^{n+1} =  {\bm \rho}^n +  \tau \sum\limits_{i=1}^s b_{i} k_i^{2}, \label{TD-rho}\\
		{\bm u}^{n+1} ={\bm u}^n +  \tau \sum\limits_{i=1}^s b_{i} k_i^{3},\label{TD-u} \\
		{\bm \phi}^{n+1} = {\bm \phi}^n +  \tau \sum\limits_{i=1}^s b_{i} k_i^{4}. \label{TD-phi}
	\end{align}
\end{sch}

\begin{thm}\label{thm:SD-CL} If the coefficients of the RK method \eqref{TDIV}-\eqref{TD-phi} satisfy the following condition
	\begin{equation}\label{RK-symplectic-condition}
		b_i a_{i j} + b_j a_{j i} = b_i b_j,\  \forall~i,j = 1,\cdots,s,
	\end{equation}
then the \textbf{Scheme \ref{scheme:QAV-RK}} conserves the following discrete invariants:
\begin{itemize}
\item The discrete mass
\begin{align}
		& \mathcal{M}_h^{n+1}=\mathcal{M}_{h}^n,\ \mathcal{M}_{h}^n=\langle|{\bm B}^{n }|^2,{\bm 1} \rangle_h,\ n=0,1,2,\cdots,.\label{QAVRKMCL}
	\end{align}
\item Two discrete quadratic invariants
\begin{align}
		& {\bm H}_{1,h}^{n+1}={\bm H}_{1,h}^{n}=0,\ \mathcal{E}_{h}^{n+1} = \mathcal{E}_{h}^n,\ {\bm H}_{1,h}^n\in\mathbb{V}_h,\ n=0,1,2,\cdots,  \label{QAVRKECL}
	\end{align}
where
\begin{align*}
&{\bm H}_{1,h}^{n}={\bm \phi}^n-|{\bm B}^{n}|^2, \label{QAVRKCphi}\\
&\mathcal{E}_{h}^n = \omega \langle{\mathcal D}_{2} {\bm B}^{n }, {\bm B}^{n }\rangle_h -\kappa \langle{\bm u}^{n }-\frac{\nu}{2} {\bm \rho}^{n }
      + \frac{q}{2} {\bm \phi}^{n }, {\bm \phi}^{n }\rangle_h - \frac{\beta}{2} \langle{\bm \rho}^{n } ,{\bm \rho}^{n }\rangle_h-
      \frac{1}{2}\langle{\bm u}^{n }, {\bm u}^{n }\rangle_h +\nu\langle{\bm u}^{n } , {\bm \rho}^{n }\rangle_h .
\end{align*}
\end{itemize}
\end{thm}

\begin{proof}
	It follows from Eq. \eqref{TD-E} that
 \begin{equation}  \label{u2diff1}
 \begin{array}{lll}
\mathcal{M}_{h}^{n+1}-\mathcal{M}_{h}^{n}&= \tau \sum\limits_{i=1}^s b_{i} \big<{\bm B}^{n},k_i^{1} \big>_h
              + \tau \sum\limits_{i=1}^s b_{i}\big< k_i^{1},{\bm B}^{n} \big>_h +\tau^2 \sum\limits_{i,j=1}^s b_{i}b_j \big<k_i^{1},k_j^{1} \big>_h.
	\end{array}
\end{equation}
Substituting ${\bm B}^n = {\bm B}_{ni} -\tau \sum\limits_{j=1}^s a_{ij} k_j^{1}$ into  \eqref{u2diff1} and notice \eqref{RK-symplectic-condition}, we have
	\begin{align}\label{u2diff2}
	\mathcal{M}_{h}^{n+1}-\mathcal{M}_{h}^{n} &= \tau \sum\limits_{i=1}^s b_{i}\langle{\bm B}_{ni},k_i^{1} \rangle_h
              + \tau \sum\limits_{i=1}^s b_{i}\langle k_i^{1},{\bm B}_{ni}\rangle_h+\tau^2\sum\limits_{i,j=1}^s (b_{i}b_j-b_ia_{ij}-b_ja_{ji})\langle k_i^{1},k_j^{1} \rangle_h\nonumber\\
              &= \tau \sum\limits_{i=1}^s b_{i}\langle{\bm B}_{ni},k_i^{1}\rangle_h
              + \tau \sum\limits_{i=1}^s b_{i}\langle k_i^{1},{\bm B}_{ni}\rangle_h.
	\end{align}
In addition, we note that
 \begin{align*}
&~~\tau \sum\limits_{i=1}^s b_{i}\langle{\bm B}_{ni},k_i^{1} \rangle_h
              + \tau \sum\limits_{i=1}^s b_{i}\langle k_i^{1},{\bm B}_{ni}\rangle_h \\
&= 2\tau \sum\limits_{i=1}^s b_{i}  {\rm Re}\langle k_i^{1}, {\bm B}_{ni}\rangle_h   \\
&= 2\tau \sum\limits_{i=1}^s b_{i} {\rm Re}   \Big(  {\rm i} \omega \langle{\mathcal D}_{2} {\bm B}_{ni},{\bm B}_{ni}\rangle_h -  {\rm i} \kappa \langle{\bm u}_{ni} - \frac{1}{2}\nu {\bm \rho}_{ni} + q {\bm \phi}_{ni}, |{\bm B}_{ni} |^2\rangle_h  \Big)  \\
&=0.
\end{align*}
Combining the above result with \eqref{u2diff2}, we obtain the discrete mass \eqref{QAVRKMCL}.

With \eqref{TD-E} and \eqref{RK-symplectic-condition}, we have
\begin{align}
 |{\bm B}^{n+1}|^2 - |{\bm B}^{n }|^2 &= {\bm B}^{n+1} \cdot \bar{{\bm B}}^{n+1}  - {\bm B}^{n}\cdot \bar{{\bm B}}^{n}  \nonumber \\
 & =  \tau \sum\limits_{i=1}^sb_{i}  ( k_i^{1} \cdot \bar{{\bm B}}^{n} )
    +\tau \sum\limits_{i=1}^sb_{i}  ( \bar{k}_i^{1} \cdot {\bm B}^{n} )
   +\tau^2 \sum\limits_{i,j=1}^s b_{i}b_{j} ( k_i^{1} \cdot \bar{k}_j^{1} ).  \label{u2diff3a}
\end{align}
Substituting ${\bm B}^n = {\bm B}_{ni} -\tau \sum\limits_{j=1}^s a_{ij} k_j^{1}$ into \eqref{u2diff3a} and using \eqref{RK-symplectic-condition}, we obtain
\begin{align}
 |{\bm B}^{n+1}|^2-|{\bm B}^{n }|^2&=\tau \sum\limits_{i=1}^sb_{i}(k_i^{1}\cdot\bar{{\bm B}}_{ni})
    +\tau \sum\limits_{i=1}^sb_{i}(\bar{k}_i^{1} \cdot {\bm B}_{ni})
   +\tau^2 \sum\limits_{i,j=1}^s (b_{i}b_{j}- b_i a_{i j} - b_j a_{j i})( k_i^{1}\cdot \bar{k}_j^{1}) \nonumber \\
  & =  \tau \sum\limits_{i=1}^sb_{i}( k_i^{1} \cdot \bar{{\bm B}}_{ni})
    +\tau \sum\limits_{i=1}^sb_{i}( \bar{k}_i^{1} \cdot {\bm B}_{ni}). \label{u2diff3c}
\end{align}
Moreover, it is clear to see
	\begin{equation}\label{u2diff3}
    \begin{array}{lll}
{\bm \phi}^{n+1} -{\bm \phi}^n=\tau \sum\limits_{i=1}^s b_i k_i^4=2\tau \sum\limits_{i=1}^s b_i {\rm Re}(k_i^1 \cdot \bar{{\bm B}}_{ni})
= \tau \sum\limits_{i=1}^sb_{i}( k_i^{1} \cdot \bar{{\bm B}}_{ni} )
    +\tau \sum\limits_{i=1}^sb_{i}(\bar{k}_i^{1} \cdot {\bm B}_{ni} ) .
	\end{array}
    \end{equation}
It follows from \eqref{u2diff3c} and \eqref{u2diff3}, together with the initial condition ${\bm \phi}^0 - |{\bm B}^{0}|^2$ that
\begin{align}
 {\bm \phi}^{n+1}- |{\bm B}^{n+1}|^2 =  {\bm \phi}^n - |{\bm B}^{n }|^2,\ n=0,1,2,\cdots,,
\end{align}
which implies the first equation of \eqref{QAVRKECL} holds.
	
Using \eqref{TD-E}, \eqref{RK-symplectic-condition} and ${\bm B}^n = {\bm B}_{ni} -\tau \sum\limits_{j=1}^s a_{ij} k_j^{1}$, we can deduce
  \begin{align*}
  &\langle{\mathcal D}_{2} {\bm B}^{n+1}, {\bm B}^{n+1}\rangle_h -\langle{\mathcal D}_{2} {\bm B}^{n}, {\bm B}^{n}\rangle_h
         = 2\tau \sum\limits_{i=1}^s b_{i} {\rm Re}\langle{\mathcal D}_{2} {\bm B}_{ni},k_i^{1}\rangle_h.
 \end{align*}
 Similarly, we have
 \begin{align*}
   &\langle{\bm u}^{n+1},{\bm \phi}^{n+1}\rangle_h-\langle{\bm u}^{n},{\bm \phi}^{n}\rangle_h
         = \tau\sum\limits_{i=1}^s b_{i} \big(\langle k_i^{3},{\bm \phi}_{ni}\rangle_h+\langle{\bm u}_{ni},k_i^{4}\rangle_h\big),\\
  &\langle{\bm \rho}^{n+1},{\bm \phi}^{n+1}\rangle_h-\langle{\bm \rho}^{n},{\bm \phi}^{n}\rangle_h
         = \tau\sum\limits_{i=1}^s b_{i} \big(\langle k_i^{2},{\bm \phi}_{ni}\rangle_h+\langle{\bm \rho}_{ni},k_i^{4}\rangle_h\big),\\
  &\langle{\bm \phi}^{n+1},{\bm \phi}^{n+1}\rangle_h-\langle{\bm \phi}^{n},{\bm \phi}^{n}\rangle_h
         =2 \tau\sum\limits_{i=1}^s b_{i}\langle{\bm \phi}_{ni},k_i^{4}\rangle_h,\\	
  &\langle{\bm\rho}^{n+1},{\bm \rho}^{n+1}\rangle_h-\langle{\bm \rho}^{n},{\bm\rho}^{n}\rangle_h
         =2\tau\sum\limits_{i=1}^s b_{i}\langle{\bm \rho}_{ni},k_i^{2}\rangle_h ,\\
  &\langle{\bm u}^{n+1},{\bm u}^{n+1}\rangle_h -\langle{\bm u}^{n},{\bm u}^{n}\rangle_h
           =2\tau \sum\limits_{i=1}^s b_{i}\langle{\bm u}_{ni},k_i^{3}\rangle_h,
\end{align*}
and
\begin{align*}
  &\langle{\bm u}^{n+1}, {\bm \rho}^{n+1}\rangle_h -\langle{\bm u}^{n},{\bm \rho}^{n}\rangle_h
         = \tau \sum\limits_{i=1}^s b_{i} \big(\langle k_i^{3},{\bm \rho}_{ni}\rangle_h +\langle{\bm u}_{ni},k_i^{2}\rangle_h\big).
  \end{align*}
Then, combining the above equations with \eqref{TDIV}, together with $k_i^{4}=2 {\rm Re} (\bar{{\bm B}}_{ni}\cdot k_i^{1})$, we can get
\begin{align*}
  &~~~	\mathcal{E}_{h}^{n+1} - \mathcal{E}_{h}^n \\
&= \tau \sum\limits_{i=1}^s b_i\Big[2\omega {\rm Re}\langle{\mathcal D}_{2} {\bm B}_{ni},k_i^{1}\rangle_h
  -\kappa \big(\langle k_i^{3},{\bm \phi}_{ni}\rangle_h+\langle{\bm u}_{ni},k_i^{4}\big>_h\big)
  +\frac{\kappa \nu }{2}  \big(\langle k_i^{2},{\bm \phi}_{ni}\big>_h +\langle{\bm \rho}_{ni},k_i^{4}\rangle_h\big)  \\
&~~~  - \kappa q\langle{\bm \phi}_{ni},k_i^{4}\rangle_h-\beta\langle{\bm \rho}_{ni},k_i^{2}\rangle_h
     -\langle{\bm u}_{ni},k_i^{3}\rangle_h
     +\nu \big(\langle k_i^{3},{\bm \rho}_{ni}\rangle_h+\langle{\bm u}_{ni},k_i^{2}\rangle_h\big)\Big] \\
&=  \tau \sum\limits_{i=1}^s b_i\Big[2{\rm Re}
   \langle\omega {\mathcal D}_{2} {\bm B}_{ni}-\kappa \big({\bm u}_{ni}-\frac{\nu }{2}{\bm \rho}_{ni}+q {\bm \phi}_{ni}\big)\cdot {\bm B}_{ni},k_i^{1}\rangle_h
 +\langle-{\bm u}_{ni} + \nu {\bm \rho}_{ni}-\kappa {\bm \phi}_{ni} , k_i^{3}\rangle_h     \\
&~~~ +\langle- \beta{\bm \rho}_{ni} + \nu {\bm u}_{ni}+\frac{1}{2}\kappa \nu {\bm \phi}_{ni} , k_i^{2}\rangle_h\Big] \\
&= \tau \sum\limits_{i=1}^s b_i \Big[-2{\rm Re}\langle{\rm i} k_i^{1},k_i^{1}\rangle_h
 +\langle-{\bm u}_{ni}+\nu{\bm\rho}_{ni}-\kappa{\bm \phi}_{ni},{\mathcal D}_{1}\big(-\beta{\bm \rho}_{ni} + \nu {\bm u}_{ni}  + \frac{1}{2}\kappa \nu {\bm \phi}_{ni} \big)\rangle_h \\
&~~~+\langle-\beta{\bm \rho}_{ni}+\nu {\bm u}_{ni}+\frac{1}{2}\kappa \nu{\bm \phi}_{ni}, {\mathcal D}_{1}\big(-{\bm u}_{ni} + \nu{\bm \rho}_{ni}-\kappa{\bm \phi}_{ni} \big)\rangle_h\Big] \\
&= 0.
\end{align*}
This immediately gives the second equation of \eqref{QAVRKECL}. This proof is completed.
\end{proof}

\begin{thm} If the coefficients of the RK method \eqref{TDIV}-\eqref{TD-phi} satisfy the condition \eqref{RK-symplectic-condition}, the \textbf{Scheme \ref{scheme:QAV-RK}} can conserve the discrete Hamiltonian energy, as follows:
	
	\begin{equation}\label{energy-relationship}
 \mathcal{H}_h^{n+1}=\mathcal{H}_h^n,\ n=0,1,2,\cdots,
\end{equation}
where
\begin{align*} 	
\mathcal{H}_h^n =\omega\langle{\mathcal D}_{2} {\bm B}^{n}, {\bm B}^{n}\rangle_h
           - \kappa\langle{\bm u}^{n}-\frac{\nu}{2}{\bm \rho}^{n }
      + \frac{q}{2}|{\bm B}^n|^2, |{\bm B}^n|^2\rangle_h- \frac{\beta}{2}\langle{\bm \rho}^{n} ,{\bm \rho}^{n}\rangle_h
      -\frac{1}{2}\langle{\bm u}^{n},{\bm u}^{n}\rangle_h+\nu\langle{\bm u}^{n} ,{\bm \rho}^{n}\rangle_h.
\end{align*}
\end{thm}
\begin{proof} According to Theorem \ref{thm:SD-CL}, the \textbf{Scheme \ref{scheme:QAV-RK}} can conserve two discrete quadratic invariants \eqref{QAVRKECL}. Then substituting ${\bm \phi}^{n+1}=|{\bm B}^{n+1}|^2$ into the second equation of \eqref{QAVRKECL}, one can deduce that
\begin{equation*}
 \mathcal{H}_h^{n+1}=\mathcal{H}_h^n,\ n=0,1,2,\cdots,
\end{equation*}
which implies that the \textbf{Scheme \ref{scheme:QAV-RK}} can conserve the discrete Hamiltonian energy. This completes the proof.
\end{proof}

\begin{thm} For any RK method, the \textbf{Scheme \ref{scheme:QAV-RK}} conserves the two discrete linear invariants, as follows:
\begin{align}\label{QAVRKCLI}
        &  \mathcal{I}_{1h}^{n+1}= \mathcal{I}_{1h}^n,\   \mathcal{I}_{2h}^{n+1}= \mathcal{I}_{2h}^n,\ n=0,1,\cdots,
        \end{align}
where
\begin{align*}
\mathcal{I}_{1h}^n=\langle{\bm \rho}^{n},{\bm 1}\rangle_h,\ \mathcal{I}_{2h}^n=\langle{\bm u}^{n} ,{\bm 1}\rangle_h.
\end{align*}
\end{thm}
\begin{proof}
According to \eqref{TD-rho}, we can obtain
\begin{align*}
&	\mathcal{I}_{1h}^{n+1}- \mathcal{I}_{1h}^n
   = \tau \sum\limits_{i=1}^s b_{i}\langle k_i^{2},{\bm 1}\rangle_h =\tau \sum\limits_{i=1}^sb_{i}\langle{\mathcal D}_{1}\big( -{\bm u}_{ni}+\nu {\bm \rho}_{ni}-\kappa {\bm \phi}_{ni} \big),{\bm 1}\rangle_h=0.
   	\end{align*}
The same procedure may be easily adapted to obtain
   	\begin{align*}
&   \mathcal{I}_{2h}^{n+1}-\mathcal{I}_{2h}^n=\tau \sum\limits_{i=1}^s b_{i}\langle k_i^{3},{\bm 1}\rangle_h= \tau \sum\limits_{i=1}^s b_{i}\langle{\mathcal D}_{1}\big(-\beta{\bm \rho}_{ni}+\nu {\bm u}_{ni}+\frac{1}{2}\kappa \nu {\bm \phi}_{ni} \big),{\bm 1}\rangle_h=0.
	\end{align*}
This completes the proof.
\end{proof}

\begin{rmk}\label{rmk-zr-3.1} If the coefficients of the RK method \eqref{TDIV}-\eqref{TD-phi} satisfy the condition \eqref{RK-symplectic-condition}, the \textbf{Scheme \ref{scheme:QAV-RK}} can rewrite into
	\begin{equation*}
		\begin{cases}
			{\bm B}_{ni} =  {\bm B}^n +  \tau \sum\limits_{j=1}^s a_{ij} k_j^{1}, \ k_i^{1} ={\rm i} \big( \omega {\mathcal D}_{2} {\bm B}_{ni}  - \kappa \big( {\bm u}_{ni} - \frac{1}{2}\nu{\bm \rho}_{ni}
                  + q {\bm \phi}_{ni} \big)\cdot  {\bm B}_{ni} \big), \\
	    	{\bm \rho}_{ni} =  {\bm \rho}^n +  \tau \sum\limits_{j=1}^s a_{ij} k_j^{2}, \ k_i^{2} =  {\mathcal D}_{1} \big( -{\bm u}_{ni} + \nu {\bm \rho}_{ni}  - \kappa {\bm \phi}_{ni} \big), \\
			{\bm u}_{ni} = {\bm u}^n +  \tau \sum\limits_{j=1}^s a_{ij}  k_j^{3}, \ k_i^{3}=  {\mathcal D}_{1}  \big( - \beta{\bm \rho}_{ni} + \nu {\bm u}_{ni}  + \frac{1}{2}\kappa \nu {\bm \phi}_{ni} \big), \\
			{\bm \phi}_{ni} =|{\bm B}^n|^2 +  2\tau \sum\limits_{j=1}^s a_{ij}{\rm Re} ( \bar{{\bm B}}_{nj} \cdot k_j^{1} ),\ i=1,2,\cdots,s,
		\end{cases}
	\end{equation*}
	and ${\bm B}^{n+1},{\bm \rho}^{n+1}$ and ${\bm u}^{n+1}$  are then updated by
	\begin{align*}
		{\bm B}^{n+1} =  {\bm B}^n +  \tau \sum\limits_{i=1}^s b_{i} k_i^{1},\\
        {\bm \rho}^{n+1} =  {\bm \rho}^n +  \tau \sum\limits_{i=1}^s b_{i} k_i^{2},\\
		{\bm u}^{n+1} ={\bm u}^n +  \tau \sum\limits_{i=1}^s b_{i} k_i^{3},
	\end{align*}
which implies that the QAV approach need introduce an auxiliary variable, but the auxiliary variable
can be eliminated in practical computations. Thus, it cannot increase additional computational costs.

\end{rmk}

\begin{rmk}
	It is proven that the Gauss methods have the same order as the underlying quadrature formula and satisfy the condition \eqref{RK-symplectic-condition} \cite{ELW06,Sanz88}. The RK coefficients of the 1-stage, 2-stage and 3-stage Gauss methods are given in Table \ref{Gaussian23}. In addition, when the RK coefficients of $s=2$ and 3 are applied to the \textbf{Scheme \ref{scheme:QAV-RK}}, we denote the obtained schemes as FPRK-2 scheme and FPRK-3 scheme, respectively.

\begin{table}[H]
\centering
\begin{tabular}{c|cc}
${c}$ & ${A}$  \\
\hline
& ${b}^{T}$ \\
\end{tabular}
=
\begin{tabular}{c|c}
$\frac{1}{2}$ &$\frac{1}{2}$ \\
\hline
 &1
\end{tabular},\
\begin{tabular}{c|cc}
${c}$ & ${A}$  \\
\hline
& ${b}^{T}$ \\
\end{tabular}
=
\begin{tabular}{c|cc}
$\frac{1}{2}-\frac{\sqrt{3}}{6}$ &$\frac{1}{4}$ & $\frac{1}{4}- \frac{\sqrt{3}}{6}$\\
$\frac{1}{2}+\frac{\sqrt{3}}{6}$ &$\frac{1}{4}+ \frac{\sqrt{3}}{6}$ &$\frac{1}{4}$ \\
\hline
                                 &$\frac{1}{2}$&$\frac{1}{2}$
\end{tabular},\\
\begin{tabular}{c|cc}
${c}$ & ${A}$  \\
\hline
& ${b}^{T}$ \\
\end{tabular}
=\begin{tabular}{c|ccc}
$\frac{1}{2}-\frac{\sqrt{15}}{10}$ &$\frac{5}{36}$ &  $\frac{2}{9}-\frac{\sqrt{15}}{15}$    &$\frac{5}{36}- \frac{\sqrt{15}}{30}$\\
$\frac{1}{2}$   &$\frac{5}{36}+ \frac{\sqrt{15}}{24}$  &  $\frac{2}{9}$  & $\frac{5}{36}- \frac{\sqrt{15}}{24}$ \\
$\frac{1}{2}+\frac{\sqrt{15}}{10}$ & $\frac{5}{36}+ \frac{\sqrt{15}}{30}$ & $\frac{2}{9}+\frac{\sqrt{15}}{15}$  & $\frac{5}{36}$ \\
\hline
& $\frac{5}{18}$ & $\frac{4}{9}$  & $\frac{5}{18}$
\end{tabular}.
\caption{\footnotesize The Gauss methods of order 2 (s=1), 4 (s=2) and 6 (s=3).}\label{Gaussian23}
\end{table}
\begin{rmk}\label{ZR-rmk-3.2} As the Gauss method of order 2 (see Table \ref{Gaussian23}) is chosen, the {\bf Scheme \ref{scheme:QAV-RK}} reduced to the CN-FP scheme \cite{WangZhang19acm}, as follows:
\begin{align}\label{CN-FP-scheme}
\left\{
\begin{aligned}
&{\rm i} \delta_t{\bm B }^n +\omega {\mathcal D}_{2} {\bm B}^{n+\frac{1}{2}}  - \kappa \big({\bm u}^{n+\frac{1}{2}}- \frac{1}{2}\nu {\bm \rho}^{n+\frac{1}{2}} + \frac{q}{2}(|{\bm B}^{n+1}|^2+|{\bm B}^{n}|^2)\big)\cdot  {\bm B}^{n+\frac{1}{2}}=0, \\
&\delta_t {\bm \rho}^n +{\mathcal D}_{1} \big({\bm u}^{n+\frac{1}{2}}-\nu {\bm \rho}^{n+\frac{1}{2}}+\frac{\kappa}{2}(|{\bm B}^{n+1}|^2+|{\bm B}^{n}|^2)\big)=0,\\
&\delta_t {\bm u}^n +{\mathcal D}_{1}  \big(\beta {\bm \rho}^{n+\frac{1}{2}}-\nu {\bm u}^{n+\frac{1}{2}}-\frac{\kappa \nu}{4}(|{\bm B}^{n+1}|^2+|{\bm B}^{n}|^2)\big)=0,\ n=0,1,2,\cdots,
	\end{aligned}
\right.
\end{align}
where
\begin{align*}
\delta_t{\bm U}^{n}=\frac{{\bm U}^{n+1}-{\bm U}^n}{\tau},\ {\bm U}^{n+\frac{1}{2}}=\frac{{\bm U}^{n+1}+{\bm U}^n}{2},\ {\bm U}^n\in\mathbb{V}_h.
\end{align*}
\end{rmk}

\end{rmk}

\section{An efficient implementation for the proposed scheme} \label{Sec-ZR-4}

In this section, we propose an efficient fixed-point iteration
solver for solving the nonlinear equations of the {\bf Scheme \ref{scheme:QAV-RK}}, which is inspired by \cite{CWJ2021cpc,ZJ2202Arix}.
For convenience, we only take the FPRK-2 scheme as an example.

For given ${\bm B}^n,\ {\bm \rho}^n$ and ${\bm u}^n$ , the FPRK-2 scheme (see Remark \ref{rmk-zr-3.1}) can be rewritten as
\begin{align}\label{f-s-1}
&k_1^1= {\rm i} \big( \omega {\mathcal D}_{2} {\bm B}_{n1}  - \kappa \big( {\bm u}_{n1} - \frac{1}{2}\nu{\bm \rho}_{n1}
                  + q {\bm \phi}_{n1} \big)\cdot  {\bm B}_{n1} \big), \\\label{f-s-1a}
&k_2^1={\rm i} \big( \omega {\mathcal D}_{2} {\bm B}_{n2}  - \kappa \big( {\bm u}_{n2} - \frac{1}{2}\nu{\bm \rho}_{n2}
                  + q {\bm \phi}_{n2} \big)\cdot  {\bm B}_{n2} \big),  \\  \label{f-s-2}
&k_1^2={\mathcal D}_{1} \big( -{\bm u}_{n1} + \nu {\bm \rho}_{n1}  - \kappa {\bm \phi}_{n1} \big),\ k_2^2={\mathcal D}_{1} \big( -{\bm u}_{n2} + \nu {\bm \rho}_{n2}  - \kappa {\bm \phi}_{n2} \big),\\\label{f-s-3}
&k_1^3={\mathcal D}_{1}\big( - \beta{\bm \rho}_{n1} + \nu {\bm u}_{n1}  + \frac{\kappa \nu}{2} {\bm \phi}_{n1} \big), \ k_2^3={\mathcal D}_{1}\big(-\beta{\bm \rho}_{n2} + \nu {\bm u}_{n2}+ \frac{\kappa \nu}{2} {\bm \phi}_{n2} \big),
\end{align}
where \begin{align}\label{f-s-4}
&{\bm B}_{n1} =  {\bm B}^n+\tau ( a_{11}k_1^1+ a_{12}k_2^1) , \
   {\bm B}_{n2} =  {\bm B}^n+\tau ( a_{21}k_1^1+ a_{22}k_2^1),\\\label{f-s-6}
&{\bm \rho}_{n1} =  {\bm \rho}^n +\tau (a_{11}k_1^2+  a_{12}k_2^2),\
  {\bm \rho}_{n2} =  {\bm \rho}^n +\tau (a_{21}k_1^2+  a_{22}k_2^2),\\\label{f-s-7}
&{\bm u}_{n1} =  {\bm u}^n + \tau (a_{11}k_1^3+  a_{12}k_2^3),\
  {\bm u}_{n2} =  {\bm u}^n + \tau \big(a_{21}k_1^3+  a_{22}k_2^3\big),\\\label{f-s-8}
&{\bm \phi}_{n1} =  |{\bm B}^n|^2 + 2\tau \Big(a_{11}\text{Re} (\bar{{\bm B}}_{n1} \cdot k_1^{1} )+  a_{12}\text{Re} ( \bar{{\bm B}}_{n2} \cdot k_2^{1} )\Big),\\
& {\bm \phi}_{n2} =  |{\bm B}^n|^2 + 2\tau\Big(a_{21}\text{Re} ( \bar{{\bm B}}_{n1} \cdot k_1^{1} )+  a_{22}\text{Re} ( \bar{{\bm B}}_{n2} \cdot k_2^{1} )\Big).
\end{align}
Then, ${\bm B}^{n+1},{\bm \rho}^{n+1}$ and ${\bm u}^{n+1}$ are updated by
\begin{align}\label{new-f-s-1}
&{\bm B}^{n+1}={\bm B}^n +\tau\sum_{i=1}^2b_ik_i^1,\ {\bm \rho}^{n+1}={\bm \rho}^n +\tau \sum_{i=1}^2b_ik_i^2,\ {\bm u}^{n+1}={\bm u}^n +\tau\sum_{i=1}^2b_ik_i^3.
\end{align}

In light of \eqref{f-s-1}-\eqref{f-s-1a} and \eqref{f-s-4}, we have
\begin{align}
&\left[
  \begin{array}{cc}\vspace{2mm}
    1-\text{i}\tau\omega a_{11} {\mathcal D}_{2} & -\text{i}\tau \omega a_{12} {\mathcal D}_{2}\\
    -\text{i}\tau \omega a_{21} {\mathcal D}_{2} & 1-\text{i}\tau \omega a_{22} {\mathcal D}_{2}\\
  \end{array}
\right]
\left[
  \begin{array}{c} \vspace{2mm}
  k_1^1\\
  k_2^1
\end{array}
  \right]=\left[
  \begin{array}{c}\vspace{2mm}
  {\bm F}_1^n\\
  {\bm F}_2^n
\end{array}
  \right],  \label{f-s-9}
\end{align}
where
\begin{align*}
{\bm F}_j^n=\text{i}\omega {\mathcal D}_{2} {\bm B}^n-\text{i}\kappa \big( {\bm u}_{nj} - \frac{1}{2}\nu{\bm \rho}_{nj}
                  + q {\bm \phi}_{nj} \big)\cdot  {\bm B}_{nj},\ j=1,2.
\end{align*}
Similarly, using \eqref{f-s-6} and \eqref{f-s-7}, the equations \eqref{f-s-2} and \eqref{f-s-3} are equivalent to
\begin{align}\label{f-s-10}
&\left[
  \begin{array}{cccc}\vspace{2mm}
    1- \nu\tau a_{11} {\mathcal D}_{1}  & - \nu\tau a_{12} {\mathcal D}_{1}
             & \tau a_{11} {\mathcal D}_{1} & \tau a_{12} {\mathcal D}_{1}  \\ \vspace{2mm}
   - \nu\tau a_{21} {\mathcal D}_{1} & 1- \nu\tau a_{22} {\mathcal D}_{1}
             & \tau a_{21} {\mathcal D}_{1} & \tau a_{22} {\mathcal D}_{1} \\ \vspace{2mm}
   \beta\tau a_{11} {\mathcal D}_{1}  & \beta\tau a_{12} {\mathcal D}_{1}
             & 1-\nu\tau a_{11} {\mathcal D}_{1} & -\nu\tau a_{12} {\mathcal D}_{1}  \\ \vspace{2mm}
   \beta\tau a_{21} {\mathcal D}_{1} & \beta\tau a_{22} {\mathcal D}_{1}
             & -\nu\tau a_{21} {\mathcal D}_{1} & 1 -\nu\tau a_{22} {\mathcal D}_{1}
  \end{array}
\right]\left[
  \begin{array}{c} \vspace{2mm}
  k_1^2\\ \vspace{2mm}
  k_2^2\\ \vspace{2mm}
  k_1^3\\ \vspace{2mm}
  k_2^3
\end{array}
  \right] =\left[
  \begin{array}{c}\vspace{2mm}
   {\mathcal D}_{1} {\bm \chi}_1^n \\ \vspace{2mm}
   {\mathcal D}_{1}{\bm \chi}_2^n \\ \vspace{2mm}
   {\mathcal D}_{1}{\bm \chi}_3^n \\ \vspace{2mm}
   {\mathcal D}_{1}{\bm \chi}_4^n
\end{array}
  \right]
\end{align}
with
\begin{align*}
&{\bm \chi}_1^n = -{\bm u}^{n} + \nu {\bm \rho}^{n}  - \kappa {\bm \phi}_{n1},\
 {\bm \chi}_2^n = -{\bm u}^{n} + \nu {\bm \rho}^{n}  - \kappa {\bm \phi}_{n2},    \\
& {\bm \chi}_3^n = - \beta {\bm \rho}^{n} + \nu{\bm u}^{n} +\frac{\kappa \nu}{2}{\bm \phi}_{n1},\
 {\bm \chi}_4^n = - \beta {\bm \rho}^{n} + \nu{\bm u}^{n} +\frac{\kappa \nu}{2} {\bm \phi}_{n2}.
\end{align*}

For the nonlinear equations \eqref{f-s-9} and \eqref{f-s-10}, we apply the fixed-point iteration strategy, as follows:
\begin{align}
&\left[
  \begin{array}{cc}\vspace{2mm}
    1-\text{i}\tau\omega a_{11} {\mathcal D}_{2} & -\text{i}\tau \omega a_{12} {\mathcal D}_{2}\\
    -\text{i}\tau \omega a_{21} {\mathcal D}_{2} & 1-\text{i}\tau \omega a_{22} {\mathcal D}_{2}\\
  \end{array}
\right]
\left[
  \begin{array}{c} \vspace{2mm}
  (k_1^1)^{l+1}\\
  (k_2^1)^{l+1}
\end{array}
  \right]=\left[
  \begin{array}{c}\vspace{2mm}
  {\bm F}_1^{n,l}\\
  {\bm F}_2^{n,l}
\end{array}
  \right],  \label{f-s-new-9}
\end{align}
and
\begin{align}\label{f-s-new-10}
&\left[
  \begin{array}{cccc}\vspace{2mm}
    1- \nu\tau a_{11} {\mathcal D}_{1}  & - \nu\tau a_{12} {\mathcal D}_{1}
             & \tau a_{11} {\mathcal D}_{1} & \tau a_{12} {\mathcal D}_{1}  \\ \vspace{2mm}
   - \nu\tau a_{21} {\mathcal D}_{1} & 1- \nu\tau a_{22} {\mathcal D}_{1}
             & \tau a_{21} {\mathcal D}_{1} & \tau a_{22} {\mathcal D}_{1} \\ \vspace{2mm}
   \beta\tau a_{11} {\mathcal D}_{1}  & \beta\tau a_{12} {\mathcal D}_{1}
             & 1-\nu\tau a_{11} {\mathcal D}_{1} & -\nu\tau a_{12} {\mathcal D}_{1}  \\ \vspace{2mm}
   \beta\tau a_{21} {\mathcal D}_{1} & \beta\tau a_{22} {\mathcal D}_{1}
             & -\nu\tau a_{21} {\mathcal D}_{1} & 1 -\nu\tau a_{22} {\mathcal D}_{1}
  \end{array}
\right]\left[
  \begin{array}{c} \vspace{2mm}
  (k_1^2)^{l+1}\\ \vspace{2mm}
  (k_2^2)^{l+1}\\ \vspace{2mm}
  (k_1^3)^{l+1}\\ \vspace{2mm}
  (k_2^3)^{l+1}
\end{array}
  \right] =\left[
  \begin{array}{c}\vspace{2mm}
   {\mathcal D}_{1} {\bm \chi}_1^{n,l} \\ \vspace{2mm}
   {\mathcal D}_{1}{\bm \chi}_2^{n,l} \\ \vspace{2mm}
   {\mathcal D}_{1}{\bm \chi}_3^{n,l} \\ \vspace{2mm}
   {\mathcal D}_{1}{\bm \chi}_4^{n,l}
\end{array}
  \right],
\end{align}
where $l=0,1,2,\cdots M-1$ and $M$ is the number of the maximum iteration step.

Using the equations \eqref{SP-D1} and \eqref{SP-D2} and let $\widetilde{\bullet}=\mathcal{F}_N\bullet$, we rewrite \eqref{f-s-new-9} and \eqref{f-s-new-10}, respectively, into
\begin{align*}
&\left[
  \begin{array}{cc}\vspace{2mm}
    1-\text{i}\tau\omega a_{11} \Lambda^{2}_j & -\text{i}\tau \omega a_{12} \Lambda^{2}_j\\
    -\text{i}\tau \omega a_{21} \Lambda^{2}_j & 1-\text{i}\tau \omega a_{22}\Lambda^{2}_j\\
  \end{array}
\right]\left[
  \begin{array}{c} \vspace{2mm}
   \Big(\widetilde{(k_1^1)^{l+1}}\Big)_j\\
  \Big(\widetilde{(k_2^1)^{l+1}}\Big)_j
\end{array}
  \right] =\left[
  \begin{array}{c}\vspace{2mm}
   \Big(\widetilde{{\bm F}_1^{n,l}}\Big)_j\\
   \Big(\widetilde{{\bm F}_2^{n,l}}\Big)_j
\end{array}
  \right]
\end{align*}
and
\begin{align*}
&\left[
  \begin{array}{cccc}\vspace{2mm}
    1- \nu\tau a_{11} \Lambda^{1}_{j}  & - \nu\tau a_{12} \Lambda^{1}_{j}
             & \tau a_{11} \Lambda^{1}_{j} & \tau a_{12} \Lambda^{1}_{j}  \\ \vspace{2mm}
   - \nu\tau a_{21} \Lambda^{1}_{j} & 1- \nu\tau a_{22} \Lambda^{1}_{j}
             & \tau a_{21} \Lambda^{1}_{j} & \tau a_{22} \Lambda^{1}_{j} \\ \vspace{2mm}
   \beta\tau a_{11} \Lambda^{1}_{j}  & \beta\tau a_{12} \Lambda^{1}_{j}
             & 1-\nu\tau a_{11} \Lambda^{1}_{j} & -\nu\tau a_{12} \Lambda^{1}_{j}  \\ \vspace{2mm}
   \beta\tau a_{21} \Lambda^{1}_{j} & \beta\tau a_{22} \Lambda^{1}_{j}
             & -\nu\tau a_{21} \Lambda^{1}_{j} & 1 -\nu\tau a_{22} \Lambda^{1}_{j}
  \end{array}
\right]\left[
  \begin{array}{c} \vspace{2mm}
   \Big(\widetilde{(k_1^2)^{l+1}}\Big)_j\\ \vspace{2mm}
  \Big(\widetilde{(k_2^2)^{l+1}}\Big)_j\\ \vspace{2mm}
  \Big(\widetilde{(k_1^3)^{l+1}}\Big)_j\\ \vspace{2mm}
  \Big(\widetilde{(k_2^3)^{l+1}}\Big)_j
\end{array}
  \right]  =\left[
  \begin{array}{c}\vspace{2mm}
   \Lambda^{1}_{j} \Big(\widetilde{{\bm\chi}_1^{n,l}}\Big)_j \\ \vspace{2mm}
   \Lambda^{1}_{j} \Big(\widetilde{{\bm\chi}_2^{n,l}}\Big)_j \\ \vspace{2mm}
   \Lambda^{1}_{j} \Big(\widetilde{{\bm\chi}_3^{n,l}}\Big)_j \\ \vspace{2mm}
   \Lambda^{1}_{j} \Big(\widetilde{{\bm \chi}_4^{n,l}}\Big)_j
\end{array}
  \right],
\end{align*}
where $j=0,1,2,\cdots,N-1$.

After we obtain $\widetilde{(k_i^1)^{M}}$,
$\widetilde{(k_i^2)^{M}}$ and $\widetilde{(k_i^3)^{M}}$, $(k_i^1)^{M}$, $(k_i^2)^{M}$  and $(k_i^3)^{M}$ are then updated by
$\mathcal{F}_N^{H}\widetilde{(k_i^1)^{M}} $, $\mathcal{F}_N^{H}\widetilde{(k_i^2)^{M}} $
and $\mathcal{F}_N^{H}\widetilde{(k_i^3)^{M}},\ i=1,2$, respectively.
Finally,  ${\bm B}^{n+1},\ {\bm \rho}^{n+1}$ and ${\bm u}^{n+1}$ are obtained by \eqref{new-f-s-1}.

In practical computation, we choose the iterative initial value $(k_i^1)^{0}={\bm B}^n$, $(k_i^2)^{0}= {\bm \rho}^n $,
$(k_i^3)^{0}= {\bm u}^n,\ i=1,2$ and the iteration terminates when the number of maximum iterative step $M=30$ is reached or the infinity norm of the error
between two adjacent iterative steps is less than $10^{-14}$, i.e.,
\begin{align*}
\mathop{\rm max}\limits_{l\leqslant i\leqslant 2}\Big \{\|({k}_i^1)^{l+1}-({k}_i^1)^l \|_{h,\infty}, \
 \| ({k}_i^2)^{l+1}-({k}_i^2)^l \|_{h,\infty},\ \| ({k}_i^3)^{l+1}-({k}_i^3)^l \|_{h,\infty}\Big\}  < 10^{-14}.
\end{align*}

\section{Numerical results}\label{Sec-ZR-5}

In this section, we will  present some numerical examples to investigate the accuracy, computational efficiency and invariants-preservation of the proposed schemes of the ZR equations \eqref{ZR1}. For brevity, in the rest of this paper, the FPRK-2 scheme and FPRK-3 scheme are only used for demonstration purposes.
In order to quantify the numerical solution, we define $L^2$- and $L^{\infty}$-norms of the error between the numerical solution and the exact solution, respectively, as
\begin{align*}
e_{B,2}(t_n)=\|{\bm B}(\cdot, t_{n})-{\bm B}^{n}\|_{h}, \
 e_{\rho,2}(t_n)=\|{\bm \rho}(\cdot, t_{n})-{\bm \rho}^{n}\|_{h}, \
 e_{u,2}(t_n)=\|{\bm u}(\cdot, t_{n})-{\bm u}^{n}\|_{h,},\\
e_{B,\infty}(t_n)=\|{\bm B}(\cdot, t_{n})-{\bm B}^{n}\|_{h,\infty}, \
 e_{\rho,\infty}(t_n)=\|{\bm \rho}(\cdot, t_{n})-{\bm \rho}^{n}\|_{h,\infty}, \
 e_{u,\infty}(t_n)=\|{\bm u}(\cdot, t_{n})-{\bm u}^{n}\|_{h,\infty}.
\end{align*}
Furthermore, we also present the relative residuals on the mass, Hamiltonian energy and two linear invariants,
defined respectively, as
\begin{align*}
RM^n=\Big|\frac{\mathcal M_h^n-\mathcal M_h^0}{\mathcal M_h^0}\Big|,\ RH^n=\Big|\frac{\mathcal H_h^n-\mathcal H_h^0}{\mathcal H_h^0}\Big|,\ RI_i^n=\Big|\frac{\mathcal{I}_{ih}^{n}-\mathcal{I}_{ih}^{0}}{\mathcal{I}_{ih}^{0}}\Big|,\ i=1,2.
\end{align*}
Furthermore, we compare the selected schemes with the CN-FP scheme \eqref{CN-FP-scheme} and the TS-FP scheme, respectively. All simulations are performed on a Win10 machine with Intel Core i7 and 32 GB using MATLAB R2015b.

\subsection{Accuracy test}

The ZR equations \eqref{ZR1} admit a solitary-wave solution given by \cite{Oliveira15,Zhao14jsc}
\begin{align}\label{ZR-solitary}
\begin{aligned}
&B(x, t)=\exp \big\{ {\rm i} \big(\lambda t+\frac{c}{2 \omega}(x-c t)+d_{0}\big)\big\} R(x-ct+x_0), \\
&\rho(x, t)=-\frac{2 c \kappa+\kappa \nu}{2 \beta-2(c+\nu)^{2}}\big|R\big(x-c t+x_{0}\big)\big|^{2}, \\
&u(x, t)=\frac{c \nu \kappa+\kappa \nu^{2}-2 \kappa \beta}{2 \beta-2(c+\nu)^{2}}\big|R\big(x-c t+x_{0}\big)\big|^{2},
	\end{aligned}
\end{align}
where
\begin{align*}
\lambda=\frac{4 \omega^{2} \eta+c^{2}}{4 \omega}, \
\zeta=q+\frac{4 c \nu \kappa + 3\kappa \nu^2 -4\kappa \beta}{4\beta-4(c+\nu)^{2}}, \
R(x)=\sqrt{\frac{2 \omega \eta}{\kappa \zeta}} \operatorname{sech}(\sqrt{\eta} x),
\end{align*}
and $x_{0}$ as well as $d_{0} $ represents the shifts of solitons in space and phase at $t=0$, respectively. Here,  we set computational domain $\Omega=[-64, 64]$, and choose the parameters $\omega=\kappa=v=c=\eta=1, \beta=7, x_{0}=2$ and $d_{0}=0$, respectively. The initial conditions are taken as the exact solution at $t=0$.

First, we test the spatial and temporal discretization error separately. In Table \ref{tabs-1}, we display the spatial $L^{2}$ and $L^{\infty}$-errors of the FPRK-2 scheme with different Fourier nodes at $T=4$, respectively, where a fixed time step size $\tau= 10^{-3}$ is chosen.
It is clear to observe that the spatial errors converge exponentially. We note that the spatial $L^{2}$ and $L^{\infty}$-errors provided by the FPRK-3 scheme with different Fourier nodes at $T=4$ are analogue to the Table \ref{tabs-1}, which is omitted for the sake of brevity.

\begin{table}[ht!]
\caption{Spatial $L^{2}$- and $L^{\infty}$-errors of the FPRK-2 scheme at $T=4$ for the ZR equations \eqref{ZR1} with different Fourier nodes.}\label{tabs-1}
\setlength{\tabcolsep}{3mm}
\centering
\begin{tabular}{cccccccc} \toprule
 Scheme           &                 & $N = 128$ & $N=256 $  & $N=512$ & $N=1024$   \\[0.1ex]
\midrule
                 &  $e_{B,2}(t_n=4)$     & 2.228e-01  & 4.294e-04  &   3.194e-09   &    1.520e-13 \\[1ex]
FPRK-$2$ scheme  &  $e_{\rho,2}(t_n=4)$  & 9.319e-02  & 1.422e-03  &   8.907e-08  &     1.604e-13\\[1ex]
                 &  $e_{u,2}(t_n=4)$     & 1.898e-01  & 2.865e-03  &   1.854e-07   &     5.304e-13 \\[1ex]
\midrule
                 &  $e_{B,\infty}(t_n=4)$     & 6.263e-02  &    1.045e-04  &   1.421e-09  &     9.783e-14\\[1ex]
FPRK-$2$ scheme  &  $e_{\rho,\infty}(t_n=4)$  & 5.363e-02  &    7.768e-04  &   5.328e-08   &     1.261e-13 \\[1ex]
                 &  $e_{u,\infty}(t_n=4)$     & 1.251e-01  &    1.982e-03  &   1.239e-07  &     4.539e-13 \\[1ex]
\bottomrule
  \end{tabular}
  \end{table}

To test the temporal discretization errors of the numerical schemes, we fix the
Fourier node 2048 so that spatial errors play no role here. The temporal $L^2$- and $L^{\infty}$-errors of the FPRK-2 scheme and FPRK-3 scheme with different time step sizes are  summarized in Tables \ref{tabt1} and \ref{tabt2}, respectively, which shows the FPRK-2 scheme and FPRK-3 scheme are fourth-order and sixth-order accurate in time, respectively.

\begin{table}[!ht]
\caption{Temporal $L^2$- and $L^{\infty}$-errors of the FPRK-2 scheme at $T=4$ for the ZR equations \eqref{ZR1} with different time step sizes.}\label{tabt1}
\setlength{\tabcolsep}{3mm}
\centering
\begin{tabular}{cccccccc} \toprule
 $\tau$                & 1/10    & 1/20& 1/40  & 1/80 \\
\midrule
                   $e_{B,2}(t_n=4)$     &1.571e-05      &9.844e-07    &6.157e-08    &3.849e-09    \\
                   rate                 & -              &3.996         & 3.999      & 4.000         \\
\midrule
                   $e_{\rho,2}(t_n=4)$  &  1.571e-05    &9.919e-07     &6.215e-08   &3.887e-09    \\
                   rate                 & -              &3.986          &3.996        & 3.999      \\
\midrule
                   $e_{u,2}(t_n=4)$     & 5.210e-05     &3.273e-06    &2.049e-07    &1.281e-08       \\
                   rate                 & -              & 3.992        &3.998         & 4.000       \\
\bottomrule
\midrule
                   $e_{B,\infty}(t_n=4)$      & 1.049e-05    &6.493e-07    &4.043e-08     &2.533e-09     \\
                   rate                       & -             & 4.014         &   4.006        &4.000        \\
\midrule
                   $e_{\rho,\infty}(t_n=4)$   & 1.234e-05    &7.872e-07     & 4.949e-08    &3.093e-09     \\
                   rate                       & -             &3.970          &3.992          & 3.998      \\
\midrule
                   $e_{u,\infty}(t_n=4)$      & 4.556e-05    &2.864e-06     &1.793e-07      &1.121e-08     \\
                   rate                       & -             &3.992          &3.998            & 4.000        \\
\bottomrule
  \end{tabular}
\end{table}


\begin{table}[!ht]
\caption{Temporal $L^2$- and $L^{\infty}$-errors of the FPRK-3 scheme at $T=4$ for the ZR equations \eqref{ZR1} with different time step sizes.}\label{tabt2}
\setlength{\tabcolsep}{3mm}
\centering
\begin{tabular}{ccccccc} \toprule
$\tau$                                     &1/20 & 1/40  & 1/80   \\
\midrule
             $e_{B,2}(t_n=4)$           & 4.346e-10   & 6.010e-12  &9.207e-14    \\
             rate                       & -            & 6.176       &6.028          \\
\midrule
             $e_{\rho,2}(t_n=4)$        &5.571e-10    &8.602e-12   &1.345e-13   \\
             rate                       & -            &  6.017       &5.999       \\
\midrule
             $e_{u,2}(t_n=4)$           & 1.066e-09   &1.677e-11  &2.634e-13    \\
             rate                       & -            &  5.990     & 5.992         \\
\bottomrule
\midrule
             $e_{B,\infty}(t_n=4)$      &3.522e-10 &5.346e-12  & 7.997e-14       \\
             rate                       & -         &6.042       & 6.063          \\
\midrule
             $e_{\rho,\infty}(t_n=4)$   &6.493e-10 &9.865e-12  & 1.545e-13     \\
             rate                       & -         & 6.041      &  5.996         \\
\midrule
              $e_{u,\infty}(t_n=4)$     & 9.499e-10 &1.489e-11  &2.334e-13     \\
              rate                      & -         &   5.996     & 5.995          \\
\bottomrule
  \end{tabular}
\end{table}

Then, we investigated the numerical $L^{\infty}$-errors of $B$ at $T=4$ versus the CPU times using the FPRK-2 scheme, the FPRK-3 scheme, the TS-FP scheme and the CN-FP scheme with the Fourier node 2048. The results are summarized in Table \ref{Error-CPU}. As illustrated, for a given global error, the proposed high-order energy-preserving schemes spend much less CPU time than the TS-FP scheme and the CN-FP scheme, and the TS-FP scheme is much cheaper than the CN-FP scheme.
\begin{table}[!ht]
\caption{The $L^{\infty}$-errors of $B$ (i.e., $e_{B,\infty}(t_n=4)$) versus total CPU time using the different numerical
schemes solving the ZR equations \eqref{ZR1}.}\label{Error-CPU}
\setlength{\tabcolsep}{3mm}
\centering
 \begin{tabular}{ccccccc} \toprule
             Scheme           &$\tau$      & $L^{\infty}$-error  & CPU time (s)     \\
\midrule
             FPRK-2 scheme     &$1.25\times 10^{-2}$ & 2.533e-09& 2.55       \\
\midrule
             FPRK-3 scheme     &$6.25\times 10^{-2}$ &1.360e-09&2.20      \\
\midrule
             TS-FP scheme     &$2\times 10^{-5}$ &2.351e-09&111.01  \\
\midrule
             CN-FP scheme     &$4\times 10^{-5}$&2.104e-09&161.75 \\
\bottomrule
  \end{tabular}
\end{table}

Subsequently, we first treat the numerical solution of the FPRK-3 scheme on a finer grid: the time step $\tau=0.01$ and Fourier node 1024 at $T=1000$ as the ``reference solution".  We then investigate the robustness of the proposed schemes in simulating evolution of the soliton at $T=1000$ as a large time step $\tau= 0.25$, and the results are summarized in Figures \ref{Rubost-TS-FP}-\ref{Rubost-FPRK-3}, which can observed that the computational results provided by the TS-FP scheme and the CN-FP scheme are wrong (see Figures  \ref{Rubost-TS-FP} and \ref{Rubost-CN-FP}) and our new schemes simulate the soliton well (see Figures \ref{Rubost-FPRK-2} and \ref{Rubost-FPRK-3}). Moreover, we also investigate the robustness of the CN-FP scheme and the TS-FP scheme on a finer time step in long time computations and the results are summarized in Figures \ref{nRubost-TS-FP} and \ref{nRubost-CN-FP}. It is clear to see that for a finer time step, the CN-FP scheme can simulate the evolution of the soliton well, while the numerical solution provided by the TS-FP scheme is wrong because of the explicit discretezation  for the nonlinear terms of the last two equations of \eqref{ZR1}.  We note the ``reference solution" is also obtained by the FPRK-3 scheme with the time step $\tau=0.01$ and the Fourier node 1024 at $T=2000$.

Finally, we investigate the residuals on the three invariants: mass, Hamiltonian energy and two linear invariants from $t=0$ and $t=1000$ using the FPRK-2 scheme and FPRK-3 scheme, respectively and the results summarized in Figure \ref{Exa1-errors}. It is clear to see that the proposed schemes can conserve the three invariants exactly, which agrees with the theoretical analysis.

\begin{figure}[h!]
\begin{minipage}[t]{0.33\linewidth}
\centering
\includegraphics[height=4.5cm,width=5cm]{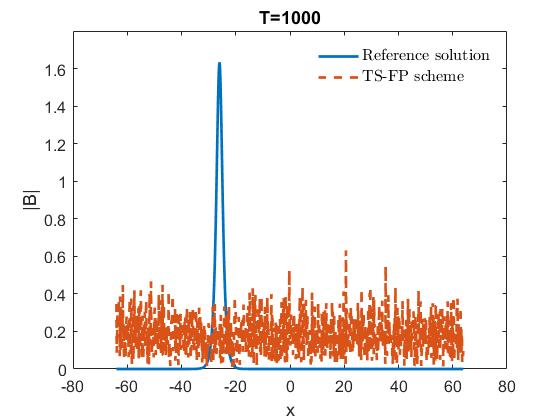}
\end{minipage}%
\begin{minipage}[t]{0.33\linewidth}
\centering
\includegraphics[height=4.5cm,width=5cm]{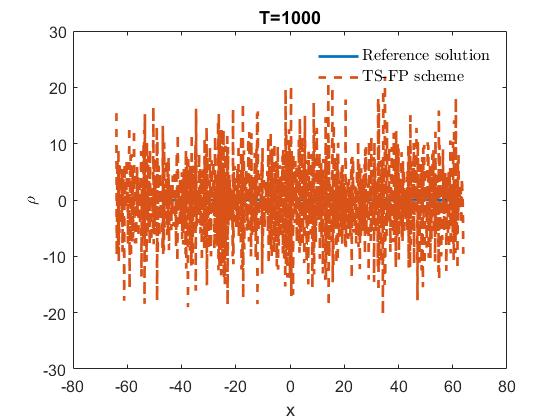}
\end{minipage}
\begin{minipage}[t]{0.33\linewidth}
\centering
\includegraphics[height=4.5cm,width=5cm]{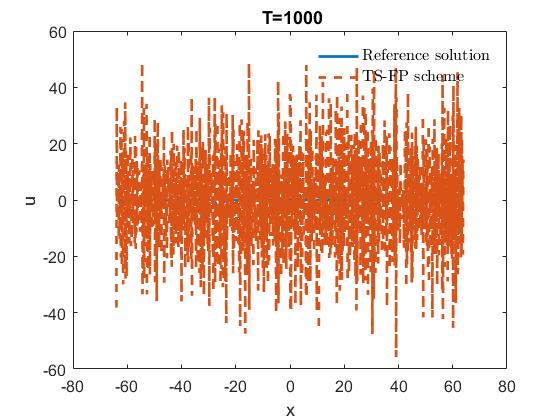}
\end{minipage}%
\vspace{-3mm}
\caption{The profile of $B,\ \rho$ and $u$ at $T = 1000$ provided by the TS-FP scheme for the ZR equations \eqref{ZR1} with the time step $\tau = 0.25$ and the Fourier node 1024.}\label{Rubost-TS-FP}
\end{figure}

\begin{figure}[H]
\begin{minipage}[t]{0.33\linewidth}
\centering
\includegraphics[height=4.5cm,width=5cm]{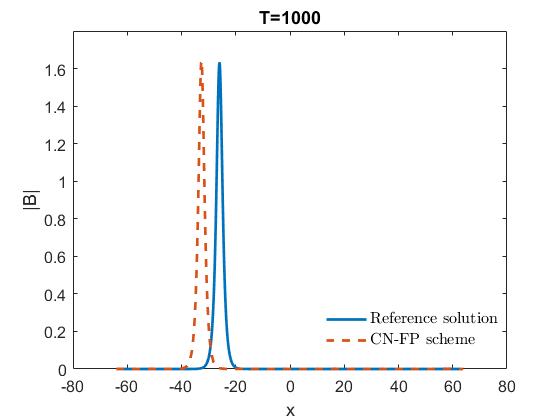}
\end{minipage}%
\begin{minipage}[t]{0.33\linewidth}
\centering
\includegraphics[height=4.5cm,width=5cm]{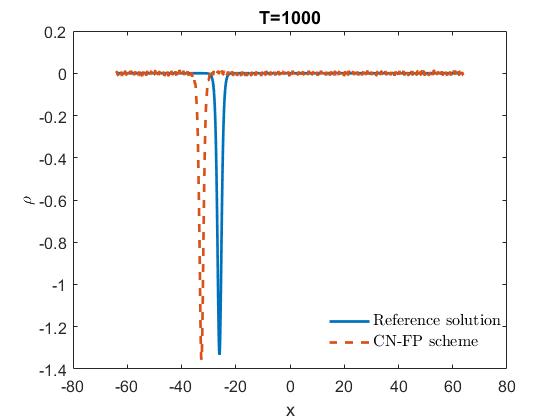}
\end{minipage}
\begin{minipage}[t]{0.33\linewidth}
\centering
\includegraphics[height=4.5cm,width=5cm]{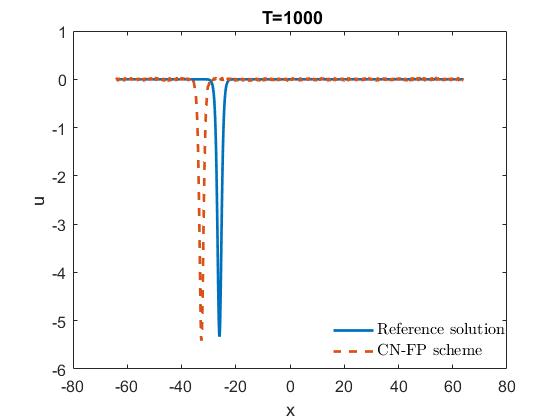}
\end{minipage}%
\vspace{-3mm}
\caption{The profile of $B,\ \rho$ and $u$ at $T = 1000$ provided by the CN-FP scheme for the  ZR equations \eqref{ZR1} with the time step $\tau = 0.25$ and the Fourier node 1024.}\label{Rubost-CN-FP}
\end{figure}

\begin{figure}[h!]
\begin{minipage}[t]{0.33\linewidth}
\centering
\includegraphics[height=4.5cm,width=5cm]{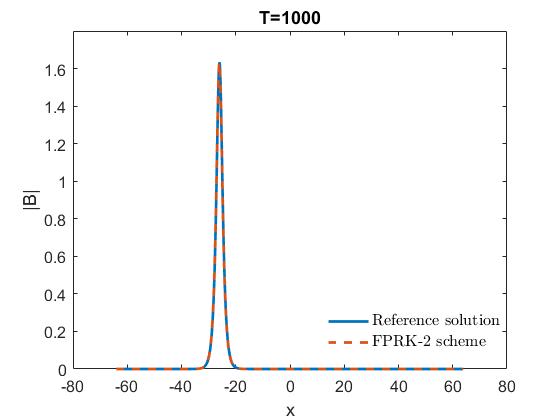}
\end{minipage}%
\begin{minipage}[t]{0.33\linewidth}
\centering
\includegraphics[height=4.5cm,width=5cm]{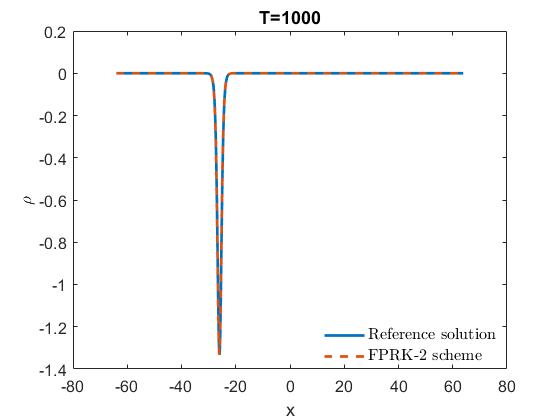}
\end{minipage}
\begin{minipage}[t]{0.33\linewidth}
\centering
\includegraphics[height=4.5cm,width=5cm]{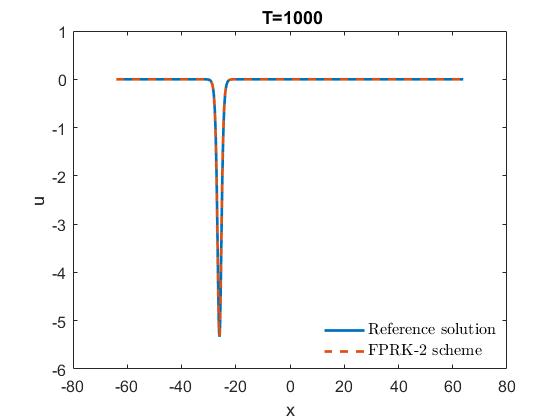}
\end{minipage}%
\vspace{-3mm}
\caption{The profile of $B,\ \rho$ and $u$ at $T = 1000$ provided by the FPRK-2 scheme for the  ZR equations \eqref{ZR1} with the time step $\tau = 0.25$ and the Fourier node 1024.}\label{Rubost-FPRK-2}
\end{figure}

\begin{figure}[h!]
\begin{minipage}[t]{0.33\linewidth}
\centering
\includegraphics[height=4.5cm,width=5cm]{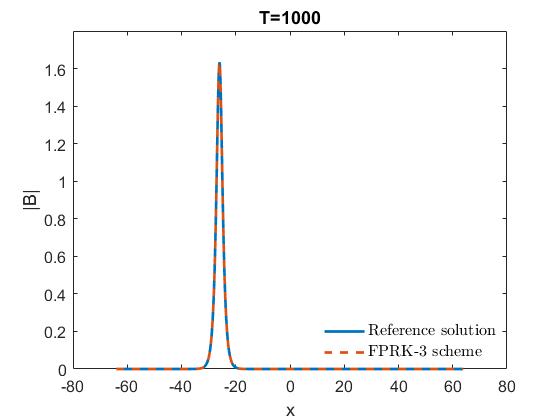}
\end{minipage}%
\begin{minipage}[t]{0.33\linewidth}
\centering
\includegraphics[height=4.5cm,width=5cm]{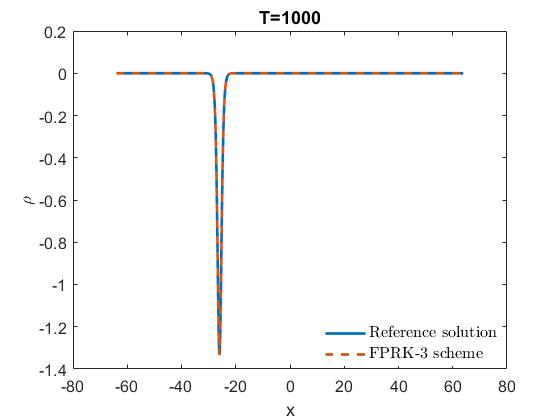}
\end{minipage}
\begin{minipage}[t]{0.33\linewidth}
\centering
\includegraphics[height=4.5cm,width=5cm]{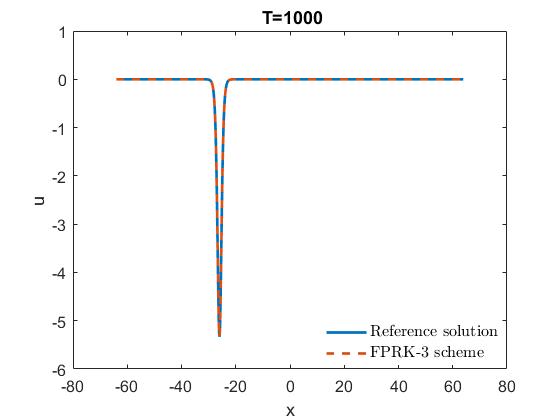}
\end{minipage}%
\vspace{-3mm}
\caption{The profile of $B,\ \rho$ and $u$ at $T = 1000$ provided by the FPRK-3 scheme for the ZR equations \eqref{ZR1} with the time step $\tau = 0.25$ and the Fourier node 1024.}\label{Rubost-FPRK-3}
\end{figure}

\begin{figure}[h!]
\begin{minipage}[t]{0.33\linewidth}
\centering
\includegraphics[height=4.5cm,width=5cm]{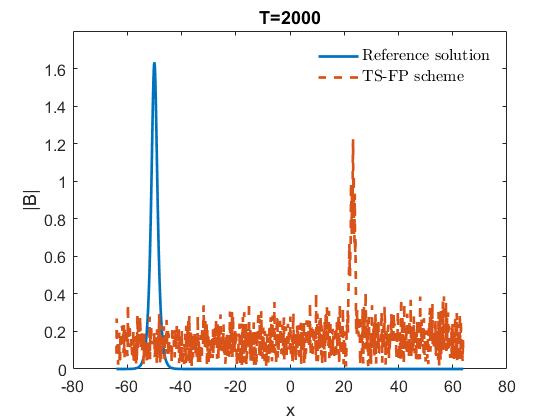}
\end{minipage}%
\begin{minipage}[t]{0.33\linewidth}
\centering
\includegraphics[height=4.5cm,width=5cm]{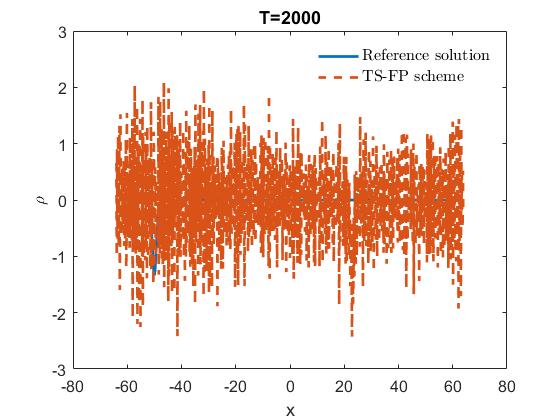}
\end{minipage}
\begin{minipage}[t]{0.33\linewidth}
\centering
\includegraphics[height=4.5cm,width=5cm]{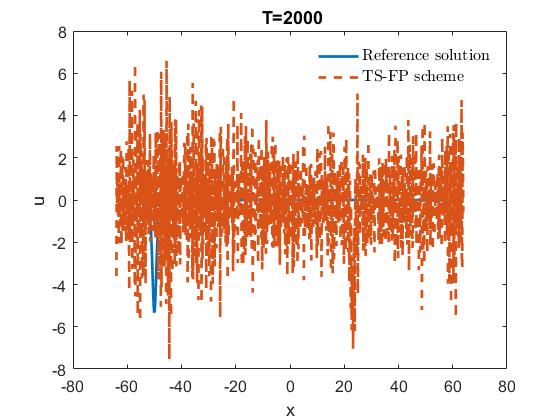}
\end{minipage}%
\vspace{-3mm}
\caption{The profile of $B,\ \rho$ and $u$ at $T = 2000$ provided by the TS-FP scheme for the ZR equations \eqref{ZR1} with the time step $\tau = 0.008$ and the Fourier node 1024.}\label{nRubost-TS-FP}
\end{figure}

\begin{figure}[H]
\begin{minipage}[t]{0.33\linewidth}
\centering
\includegraphics[height=4.5cm,width=5cm]{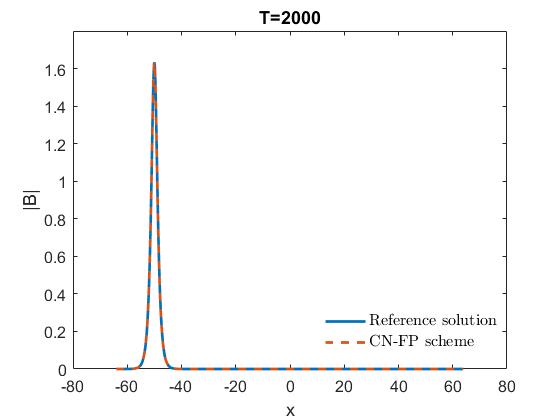}
\end{minipage}%
\begin{minipage}[t]{0.33\linewidth}
\centering
\includegraphics[height=4.5cm,width=5cm]{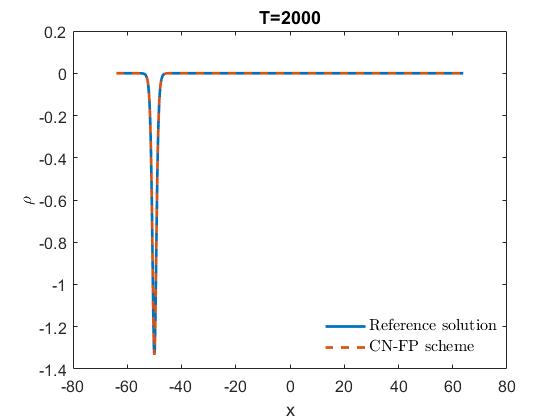}
\end{minipage}
\begin{minipage}[t]{0.33\linewidth}
\centering
\includegraphics[height=4.5cm,width=5cm]{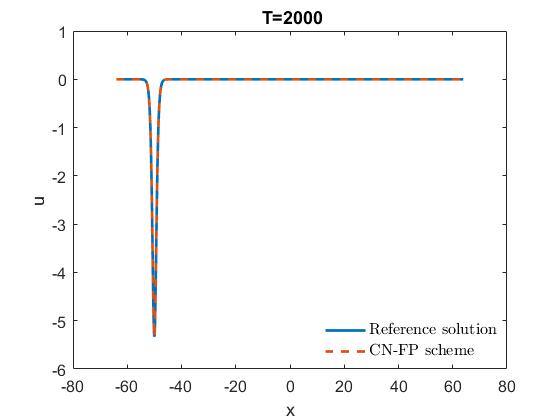}
\end{minipage}%
\vspace{-3mm}
\caption{The profile of $B,\ \rho$ and $u$ at $T = 2000$ provided by the CN-FP scheme for the ZR equations \eqref{ZR1} with the time step $\tau = 0.05$ and the Fourier node 1024.}\label{nRubost-CN-FP}
\end{figure}

\begin{figure}[H]
\centering
\subfigure{
\begin{minipage}[t]{0.5\linewidth}
\centering
\includegraphics[height=4cm,width=6.0cm]{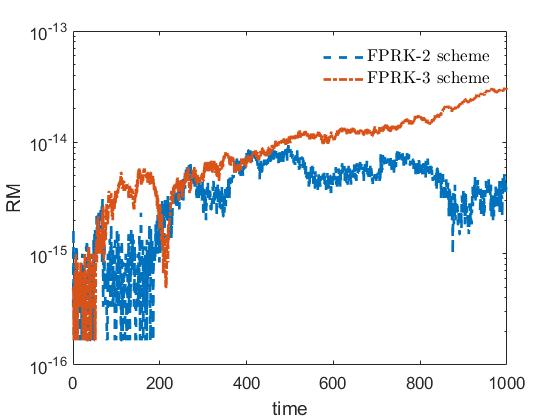}
\end{minipage}%
}%
\subfigure{
\begin{minipage}[t]{0.5\linewidth}
\centering
\includegraphics[height=4cm,width=6.0cm]{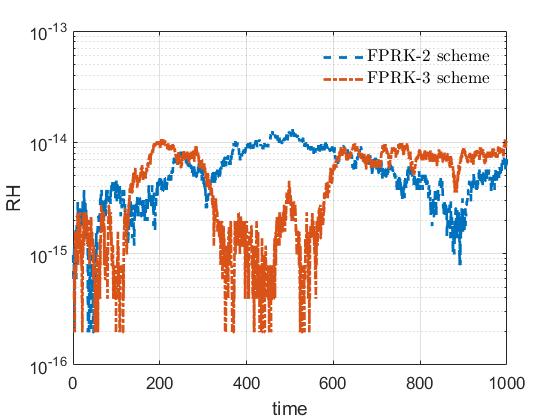}
\end{minipage}%
}%
\\
\subfigure{
\begin{minipage}[t]{0.5\linewidth}
\centering
\includegraphics[height=4cm,width=6.0cm]{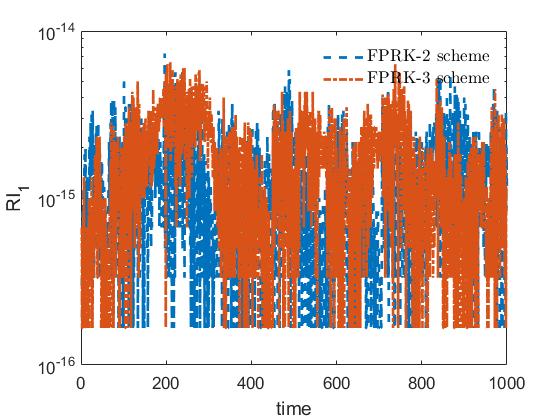}
\end{minipage}%
}%
\subfigure{
\begin{minipage}[t]{0.5\linewidth}
\centering
\includegraphics[height=4cm,width=6.0cm]{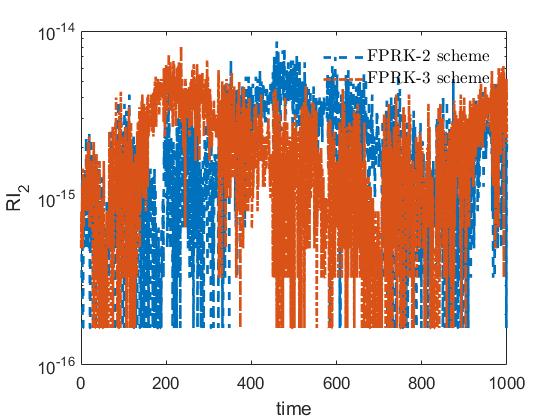}
\end{minipage}%
}%
\\
\caption{The relative residuals on the mass, Hamiltonian energy and two linear invariants over the time interval $t\in[0,1000]$ with the time step $\tau=0.01$ and the Fourier node 1024.}\label{Exa1-errors}
\end{figure}

\subsection{Interactions of solitons}
 The initial data are then chosen as  (cf. \cite{Zhao14jsc})
\begin{align*}
\label{ZR-collisions}
\begin{aligned}
&B_{0}(x)=\exp \big\{ {\rm i} \big( \frac{c_{+}}{2 \omega} x +d_{0}\big)\big\} R_{+}\big(x +x_{0+}\big) + \exp \big\{ {\rm i} \big( \frac{c_{-}}{2 \omega} x +d_{0}\big)\big\} R_{-}\big(x +x_{0-}\big), \\
&\rho_{0}(x)=-\frac{2 c_{+} \kappa+\kappa \nu}{2 \beta-2\big(c_{+}+\nu\big)^{2}}\big|R_{+}\big(x+x_{0+}\big)\big|^{2}-\frac{2 c_{-} \kappa+\kappa \nu}{2 \beta-2\big(c_{-}+\nu\big)^{2}}\big|R_{-}\big(x+x_{0-}\big)\big|^{2}, \\
&u_{0}(x)=\frac{c_{+} \nu \kappa+\kappa \nu^{2}-2 \kappa \beta}{2 \beta-2\big(c_{+}+\nu\big)^{2}}\big|R_{+}\big(x+x_{0 +}\big)\big|^{2} + \frac{c_{-}\nu \kappa+\kappa \nu^{2}-2 \kappa \beta}{2 \beta-2\big(c_{-}+\nu\big)^{2}}\big|R_{-}\big(x+x_{0 -}\big)\big|^{2},  \   x \in\Omega,
	\end{aligned}
\end{align*}
where $c_{ \pm}$ and $x_{0 \pm}$ are different velocities and different  initial locations of the two solitary waves, respectively,  $\eta_{\pm}>0$, and
\begin{align*}
\lambda_{\pm}=\frac{4 \omega^{2} \eta_{\pm}+c_{\pm}^{2}}{4 \omega}, \  \zeta_{\pm}=q+\frac{4 c_{\pm} \nu \kappa+3 \kappa \nu^{2}-4\kappa \beta}{4\beta-4(c_{\pm}+\nu)^{2}}, \
R_{\pm}(x)=\sqrt{\frac{2\omega\eta_{\pm}}{\kappa\zeta_{\pm}}}\operatorname{sech}\big(\sqrt{\eta_{\pm}} x\big).
\end{align*}

We take the parameters, respectively, as:
\begin{itemize}
\item Case I: High velocity case:
\begin{align*}
\omega &=\eta_{\pm}=1, \   \kappa=2, \   \nu=0.2, \   c_{\pm}=\pm 8, \   x_{0 \pm}=\pm 8, \   \beta=75, \   d_{0}=0.
\end{align*}

\item Case II: Intermediate velocity case:
\begin{align*}
\omega & =\eta_{\pm}=1, \   \kappa=3, \   \nu=0.2, \   c_{\pm}=\pm 1.5, \   x_{0 \pm}=\pm 9,\    \beta=12, \   d_{0}=0.
\end{align*}

\item Case III: Small velocity case:
\begin{align*}
\omega & =\eta_{\pm}=1,\    \kappa=1, \   \nu=0.5, \   c_{+}=0, \   c_{-}=-0.5, \   x_{0+}=-8, \   x_{0-}=-26, \\
\beta & =3, \   d_{0}=0.
\end{align*}
\end{itemize}

Furthermore, all numerical calculations are performed using the numerical schemes with the time step $\tau=1/200$ and the spatial mesh size $h={1}/{8}$. The profiles and contour plots of collisions for $B,\rho$ as well as $u$ for the Cases I-III provided by the FPRK-2 scheme are shown in Figures \ref{interac1}-\ref{interac3},
respectively. We can observe that all the collisions between two solitons are not elastic. Moreover, From Figures \ref{interac2}-\ref{interac3}, we can observe that some little dispersive waves are produced throughout the collision or following the collision, which implies that the ZR equations \eqref{ZR1} is a non-integrable system \cite{Zhao14jsc,Oruc21AMC,WangZhang19acm}. The results are completely in good agreement with those presented in literatures \cite{Oruc21AMC,Zhao14jsc,WangZhang19acm}. Here, we point out that the profile and contour plots of collisions for $B,\rho$ as well as $u$ calculated using the FPRK-3 scheme for the Cases I-III are similar to the Figures \ref{interac1}-\ref{interac3}, thus for brevity, we omit them. In Figures \ref{Exa2-caseI-errors}-\ref{Exa2-caseIII-errors}, we investigate the residuals on the three invariants: mass, Hamiltonian energy and two linear invariants using the FPRK-2 scheme and FPRK-3 scheme, respectively. It is obvious that the residuals on the three invariants are preserved to
machine round-off error.

\section{Conclusion}\label{Sec-ZR-6}

In this paper, we propose a novel high-order energy-preserving scheme extending the classical Crank-Nicolson proposed method in \cite{Zhao14jsc,WangZhang19acm} to high order.
 Based on the idea of the QAV approach, we first transform the Hamiltonian energy into a quadratic from by introducing a quadratic auxiliary variable, and the original system is then reformulated into a new equivalent system using the energy variational principle. Finally, a fully-discrete scheme is presented by using the symplectic RK method in time and the
Fourier pseudo-spectral method in space for the reformulated system. We show that the proposed scheme can preserve the three invariants \eqref{CONmass}-\eqref{CONI1} of the ZR equations \eqref{ZR1} exactly. In addition, an efficient iterative solver is presented to solve the discrete nonlinear equations of our scheme. Numerical results show that the proposed scheme has more advantages than the TS-FP scheme \cite{Zhao14jsc} and CN-FP scheme \cite{WangZhang19acm} as the large time  step is chosen. However, we should note that the proposed scheme is fully-implicit. Thus, an interesting topic for future studies is whether it is possible to construct high-order linearly
implicit schemes which can preserve the three invariants \eqref{CONmass}-\eqref{CONI1} of the ZR equations \eqref{ZR1} exactly.

\section*{Acknowledgments}
The authors would like to express sincere gratitude to the referees for their insightful comments and suggestions. The work is supported by the National Natural Science Foundation of China (Grant Nos.11901513, 12261097, 12261103), and the Yunnan Fundamental Research Projects (Grant Nos. 202101AT070208,  202001AT070066, 202101AS070044), the Natural Science Foundation of Hunan (Grant No. 2021JJ40655) and Innovation team of School of Mathematics and Statistics, Yunnan University (No.
ST20210104). The first author is in particular grateful to Prof. Weizhu Bao for fruitful discussions.

\begin{figure}[H]
\centering
\subfigure[$|B(x,t)|$]{
\begin{minipage}[t]{0.5\linewidth}
\centering
\includegraphics[height=4.5cm,width=6.0cm]{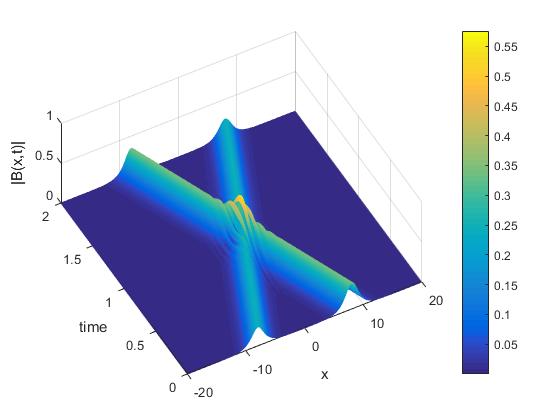}
\end{minipage}%
}%
\subfigure[$|B(x,t)|$]{
\begin{minipage}[t]{0.5\linewidth}
\centering
\includegraphics[height=4.5cm,width=6.0cm]{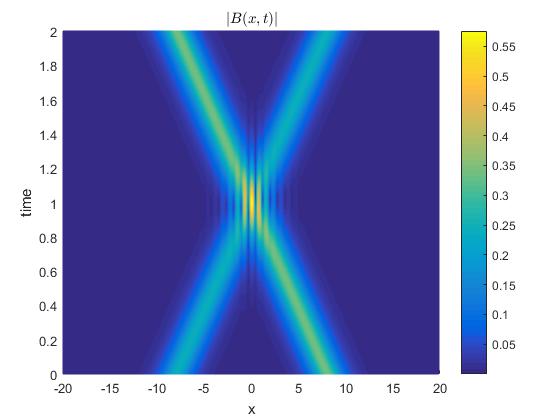}
\end{minipage}%
}%
\\
\subfigure[$\rho(x,t)$]{
\begin{minipage}[t]{0.5\linewidth}
\centering
\includegraphics[height=4.5cm,width=6.0cm]{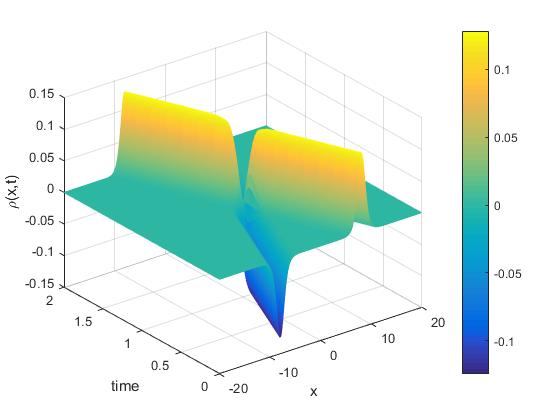}
\end{minipage}%
}%
\subfigure[$\rho(x,t)$]{
\begin{minipage}[t]{0.5\linewidth}
\centering
\includegraphics[height=4.5cm,width=6.0cm]{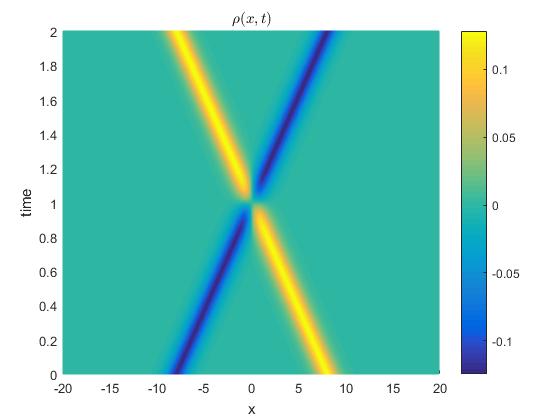}
\end{minipage}%
}%
\\
\subfigure[$u(x,t)$]{
\begin{minipage}[t]{0.5\linewidth}
\centering
\includegraphics[height=4.5cm,width=6.0cm]{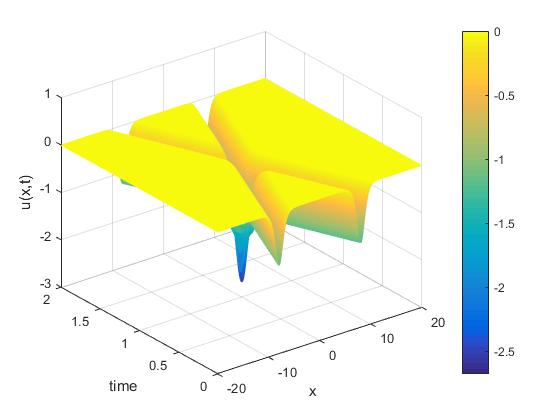}
\end{minipage}
}%
\subfigure[$u(x,t)$]{
\begin{minipage}[t]{0.5\linewidth}
\centering
\includegraphics[height=4.5cm,width=6.0cm]{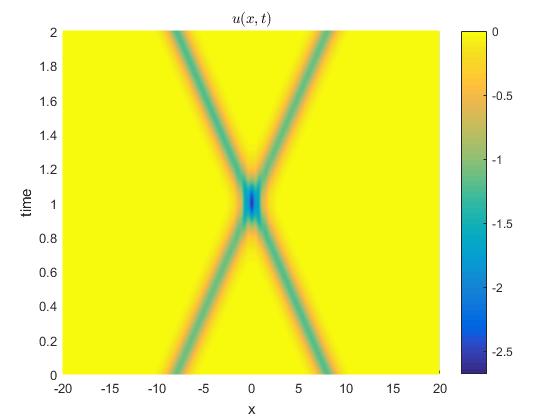}
\end{minipage}
}%
 \\
\centering
\caption{ Inelastic collision between two solitons provided by the FPRK-2 scheme for the ZR equations \eqref{ZR1} under Case I, where the computational is chosen as  $\Omega=[-20,20]$.}\label{interac1}
\end{figure}

\begin{figure}[H]
\centering
\subfigure[$|B(x,t)|$]{
\begin{minipage}[t]{0.5\linewidth}
\centering
\includegraphics[height=5.5cm,width=6.0cm]{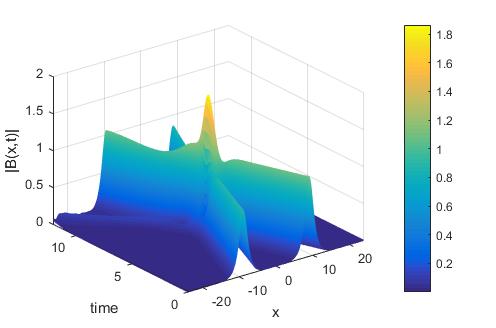}
\end{minipage}%
}%
\subfigure[$|B(x,t)|$]{
\begin{minipage}[t]{0.5\linewidth}
\centering
\includegraphics[height=5.5cm,width=6.0cm]{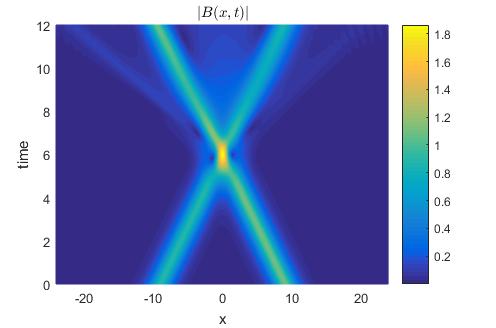}
\end{minipage}%
}%
\\
\subfigure[$\rho(x,t)$]{
\begin{minipage}[t]{0.5\linewidth}
\centering
\includegraphics[height=5.5cm,width=6.0cm]{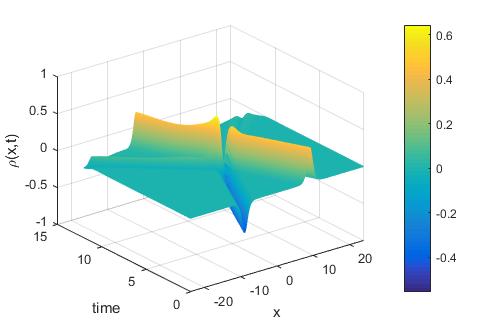}
\end{minipage}%
}%
\subfigure[$\rho(x,t)$]{
\begin{minipage}[t]{0.5\linewidth}
\centering
\includegraphics[height=5.5cm,width=6.0cm]{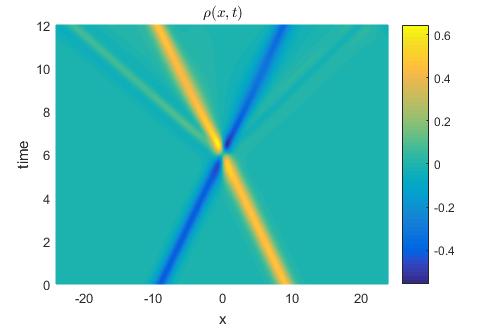}
\end{minipage}%
}%
\\
\subfigure[$u(x,t)$]{
\begin{minipage}[t]{0.5\linewidth}
\centering
\includegraphics[height=5.5cm,width=6.0cm]{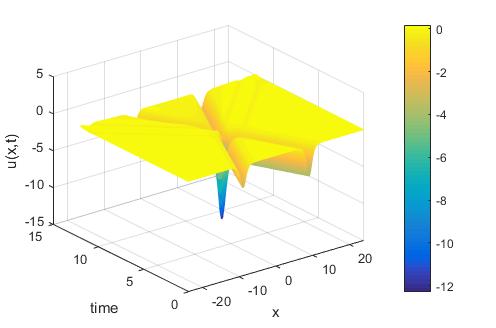}
\end{minipage}
}%
\subfigure[$u(x,t)$]{
\begin{minipage}[t]{0.5\linewidth}
\centering
\includegraphics[height=5.5cm,width=6.0cm]{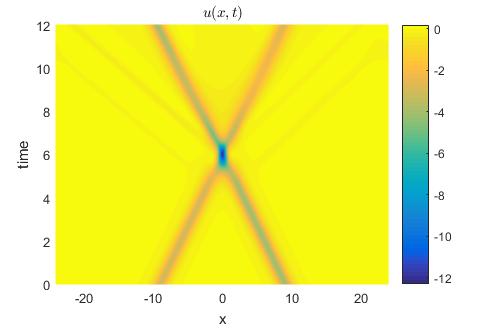}
\end{minipage}
}%
 \\
\centering
\caption{Inelastic collision between two solitons provided by the FPRK-2 scheme for the ZR equations \eqref{ZR1} under Case II, where the computational is chosen as $\Omega=[-24,24]$.}\label{interac2}
\end{figure}

\begin{figure}[H]
\centering
\subfigure[$|B(x,t)|$]{
\begin{minipage}[t]{0.5\linewidth}
\centering
\includegraphics[height=5.5cm,width=6.0cm]{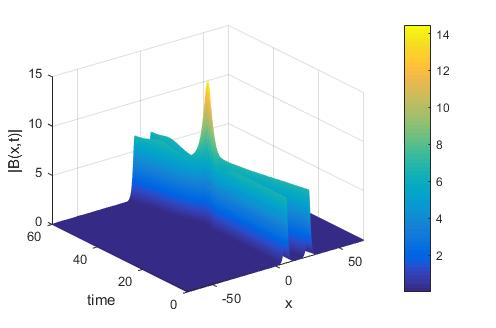}
\end{minipage}%
}%
\subfigure[$|B(x,t)|$]{
\begin{minipage}[t]{0.5\linewidth}
\centering
\includegraphics[height=5.5cm,width=6.0cm]{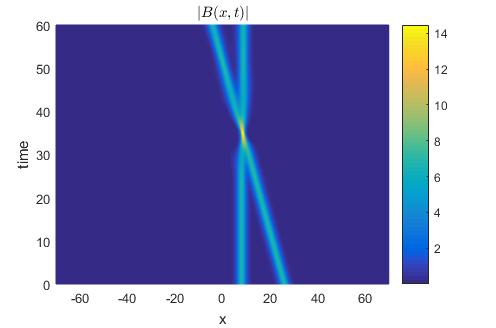}
\end{minipage}%
}%
\\
\subfigure[$\rho(x,t)$]{
\begin{minipage}[t]{0.5\linewidth}
\centering
\includegraphics[height=5.5cm,width=6.0cm]{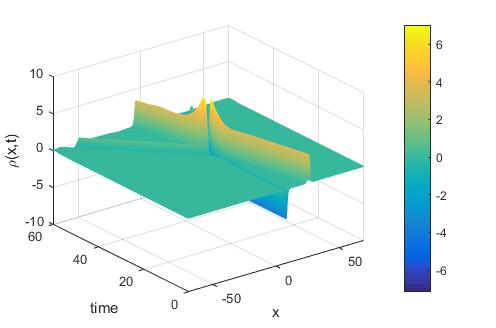}
\end{minipage}%
}%
\subfigure[$\rho(x,t)$]{
\begin{minipage}[t]{0.5\linewidth}
\centering
\includegraphics[height=5.5cm,width=6.0cm]{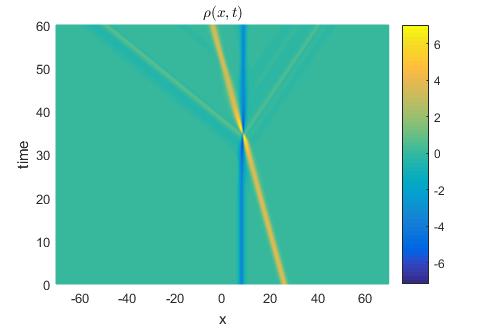}
\end{minipage}%
}%
\\
\subfigure[$u(x,t)$]{
\begin{minipage}[t]{0.5\linewidth}
\centering
\includegraphics[height=5.5cm,width=6.0cm]{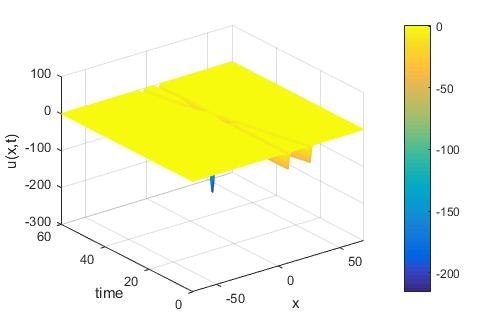}
\end{minipage}
}%
\subfigure[$u(x,t)$]{
\begin{minipage}[t]{0.5\linewidth}
\centering
\includegraphics[height=5.5cm,width=6.0cm]{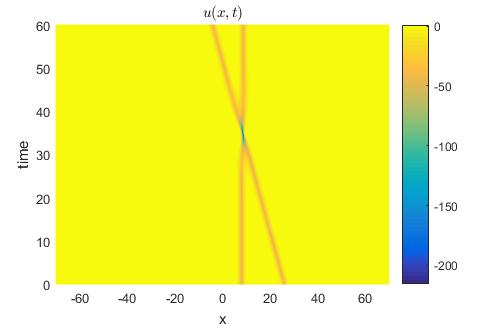}
\end{minipage}
}%
 \\
\centering
\caption{Inelastic collision between two solitons provided by the FPRK-2 scheme for the ZR equations \eqref{ZR1} under Case III, where the computational is chosen as $\Omega=[-70,70]$.}\label{interac3}
\end{figure}

\begin{figure}[H]
\centering
\subfigure{
\begin{minipage}[t]{0.5\linewidth}
\centering
\includegraphics[height=4.5cm,width=6.0cm]{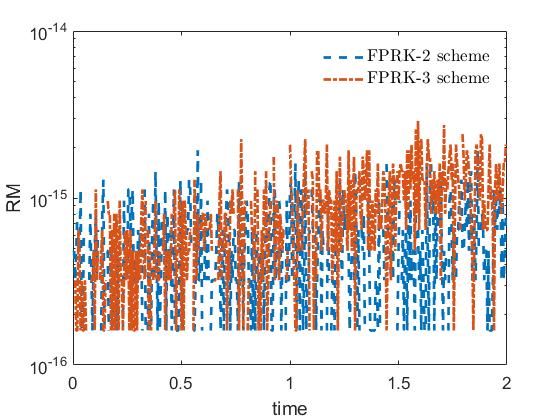}
\end{minipage}%
}%
\subfigure{
\begin{minipage}[t]{0.5\linewidth}
\centering
\includegraphics[height=4.5cm,width=6.0cm]{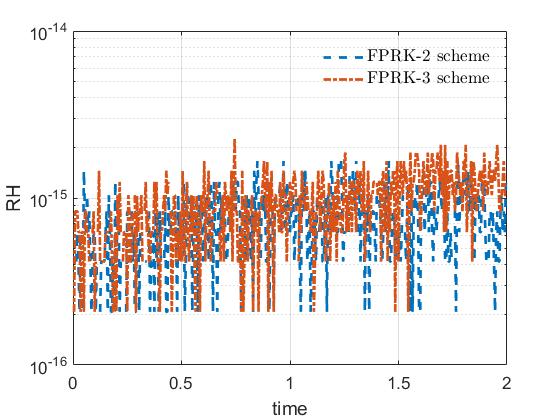}
\end{minipage}%
}%
\\
\subfigure{
\begin{minipage}[t]{0.5\linewidth}
\centering
\includegraphics[height=4.5cm,width=6.0cm]{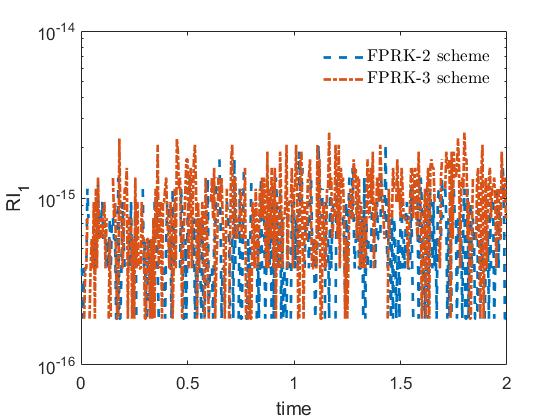}
\end{minipage}%
}%
\subfigure{
\begin{minipage}[t]{0.5\linewidth}
\centering
\includegraphics[height=4.5cm,width=6.0cm]{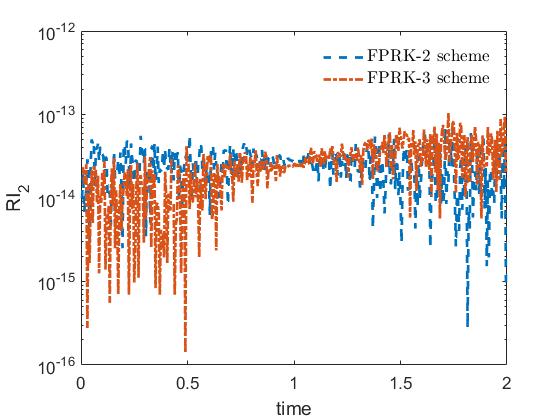}
\end{minipage}%
}%
\\
\centering
\caption{The relative residuals on the mass, Hamiltonian energy and two linear invariants over the time interval $t\in[0,2]$ under Case I.}\label{Exa2-caseI-errors}
\end{figure}

\begin{figure}[H]
\centering
\subfigure{
\begin{minipage}[t]{0.5\linewidth}
\centering
\includegraphics[height=4.5cm,width=6.0cm]{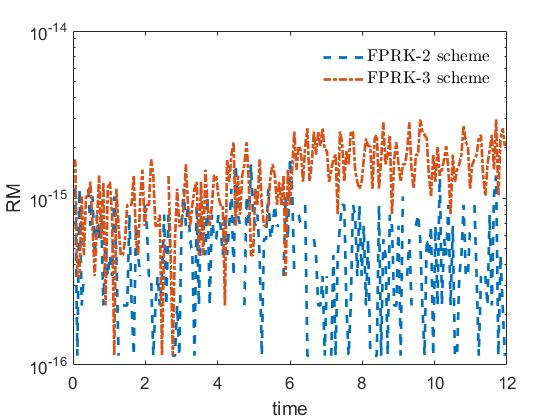}
\end{minipage}%
}%
\subfigure{
\begin{minipage}[t]{0.5\linewidth}
\centering
\includegraphics[height=4.5cm,width=6.0cm]{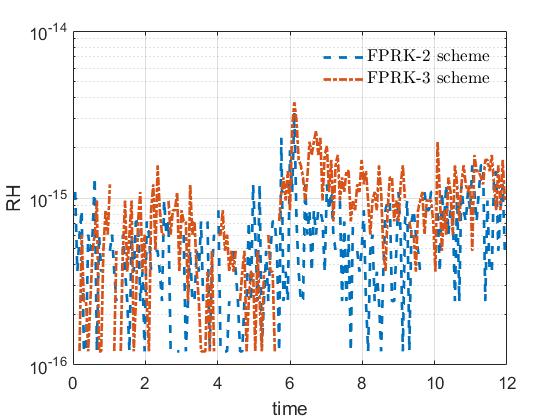}
\end{minipage}%
}%
\\
\subfigure{
\begin{minipage}[t]{0.5\linewidth}
\centering
\includegraphics[height=4.5cm,width=6.0cm]{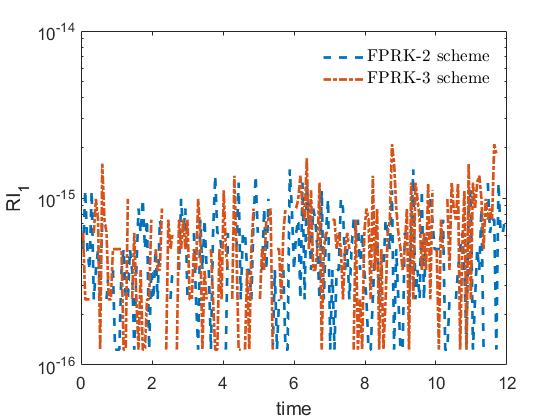}
\end{minipage}%
}%
\subfigure{
\begin{minipage}[t]{0.5\linewidth}
\centering
\includegraphics[height=4.5cm,width=6.0cm]{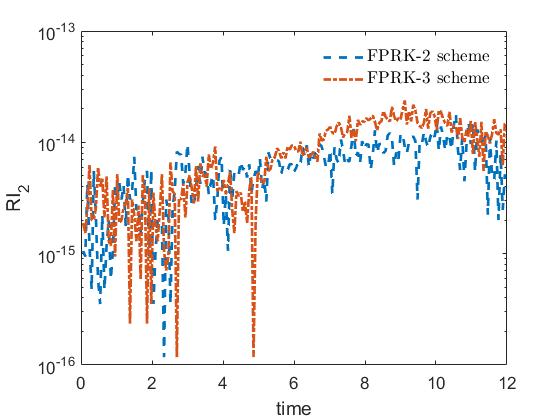}
\end{minipage}%
}%
\\
\centering
\caption{The relative residuals on the mass, Hamiltonian energy and two linear invariants over the time interval $t\in[0,12]$ under Case II.}\label{Exa2-caseII-errors}
\end{figure}

\begin{figure}[H]
\centering
\subfigure{
\begin{minipage}[t]{0.5\linewidth}
\centering
\includegraphics[height=4.5cm,width=6.0cm]{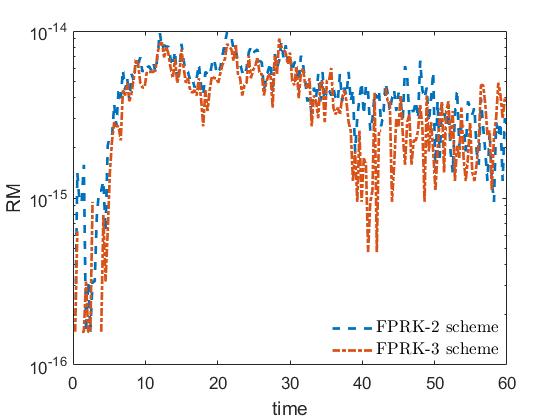}
\end{minipage}%
}%
\subfigure{
\begin{minipage}[t]{0.5\linewidth}
\centering
\includegraphics[height=4.5cm,width=6.0cm]{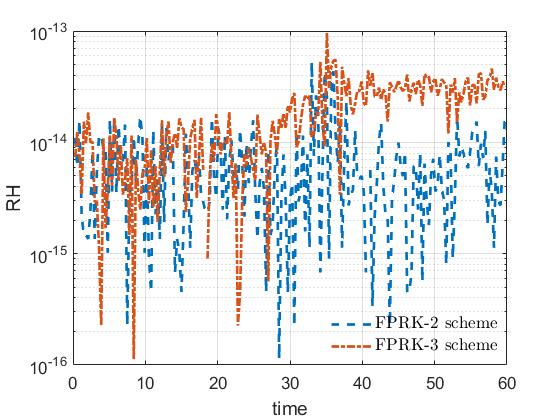}
\end{minipage}%
}%
\\
\subfigure{
\begin{minipage}[t]{0.5\linewidth}
\centering
\includegraphics[height=4.5cm,width=6.0cm]{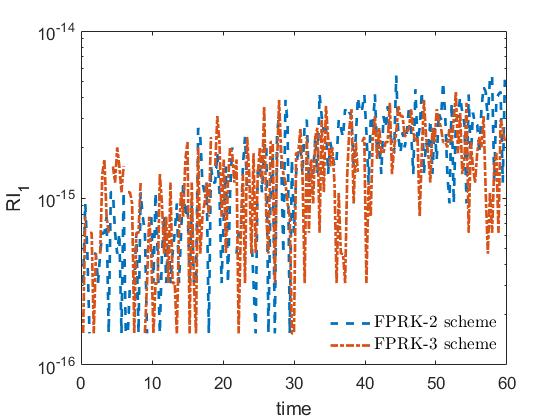}
\end{minipage}%
}%
\subfigure{
\begin{minipage}[t]{0.5\linewidth}
\centering
\includegraphics[height=4.5cm,width=6.0cm]{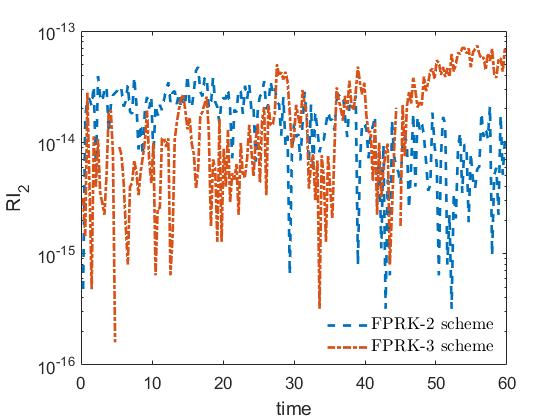}
\end{minipage}%
}%
\\
\centering
\caption{The relative residuals on the mass, Hamiltonian energy and two linear invariants over the time interval $t\in[0,60]$ under Case III.}\label{Exa2-caseIII-errors}
\end{figure}


\begin{thebibliography}{10}

\bibitem{BBCIamc18}
L.~Barletti, L.~Brugnano, G.~F. Caccia, and F.~Iavernaro.
\newblock Energy-conserving methods for the nonlinear {S}chr\"{o}dinger
  equation.
\newblock {\em Appl. Math. Comput.}, 318:3--18, 2018.

\bibitem{BI16}
L.~Brugnano and F.~Iavernaro.
\newblock {\em Line Integral Methods for Conservative Problems}.
\newblock Chapman et Hall/CRC: Boca Raton, FL, USA, 2016.

\bibitem{BIT10}
L.~Brugnano, F.~Iavernaro, and D.~Trigiante.
\newblock Hamiltonian boundary value methods (energy preserving discrete line
  integral methods).
\newblock {\em J. Numer. Anal. Ind. Appl. Math.}, 5:17--37, 2010.

\bibitem{CIZ1997-MP}
M.~Calvo, A.~Iserles, and A.~Zanna.
\newblock Numerical solution of isospectral flows.
\newblock {\em Math. Comp.}, 66:1461--1486, 1997.

\bibitem{CQ01}
J.~Chen and M.~Qin.
\newblock Multi-symplectic {F}ourier pseudospectral method for the nonlinear
  {S}chr\"{o}dinger equation.
\newblock {\em Electr. Trans. Numer. Anal.}, 12:193--204, 2001.

\bibitem{GongQuezheng21}
Y.~Chen, Y.~Gong, Q.~Hong, and C.~Wang.
\newblock A novel class of energy-preserving {Runge-Kutta} methods for the
  {K}orteweg-de {V}ries equation.
\newblock {\em Numer. Math. Theor. Meth. Appl.}, 15:768--792, 2022.

\bibitem{CH11bit}
D.~Cohen and E.~Hairer.
\newblock Linear energy-preserving integrators for {P}oisson systems.
\newblock {\em BIT}, 51:91--101, 2011.

\bibitem{Cooper1987}
G.~J. Cooper.
\newblock Stability of {R}unge-{K}utta methods for trajectory problems.
\newblock {\em IMA J. Numer. Anal.}, 7:1--13, 1987.

\bibitem{Cordero16}
J.~Cordero.
\newblock Supersonic limit for the {Z}akharov--{R}ubenchik system.
\newblock {\em J. Differ. Equ.}, 261:5260--5288, 2016.

\bibitem{CWJ2021cpc}
J.~Cui, Y.~Wang, and C.~Jiang.
\newblock Arbitrarily high-order structure-preserving schemes for the
  {G}ross-{P}itaevskii equation with angular momentum rotation.
\newblock {\em Comput. Phys. Commun.}, 261:107767, 2021.

\bibitem{GCW14CiCP}
Y.~Gong, J.~Cai, and Y.~Wang.
\newblock Multi-symplectic {F}ourier pseudospectral method for the {K}awahara
  equation.
\newblock {\em Commun. Comput. Phys.}, 16:35--55, 2014.

\bibitem{GHWW2022}
Y.~Gong, Q.~Hong, C.~Wang, and Y.~Wang.
\newblock Arbitrarily high-order energy-preserving schemes for the
  {Camassa-Holm} equation based on the quadratic auxiliary variable approach.
\newblock {\em Adv. Appl. Math. Mech. (accepted)}.

\bibitem{GZW2020jcp}
Y.~Gong, J.~Zhao, and Q.~Wang.
\newblock Arbitrarily high-order linear energy stable schemes for gradient flow
  models.
\newblock {\em J. Comput. Phys.}, 419:109610, 2020.

\bibitem{H10}
E.~Hairer.
\newblock Energy-preserving variant of collocation methods.
\newblock {\em J. Numer. Anal. Ind. Appl. Math.}, 5:73--84, 2010.

\bibitem{ELW06}
E.~Hairer, C.~Lubich, and G.~Wanner.
\newblock {\em Geometric Numerical Integration: Structure-Preserving Algorithms
  for Ordinary Differential Equations}.
\newblock Springer-Verlag, Berlin, 2nd edition, 2006.

\bibitem{JiZhang19ijcm}
B.~Ji, L.~Zhang, and X.~Zhou.
\newblock Conservative compact difference scheme for the
  {Z}akharov--{R}ubenchik equations.
\newblock {\em Int. J. Comput. Math.}, 96:537--556, 2019.

\bibitem{JCQSjsc2022}
C.~Jiang, J.~Cui, X.~Qian, and S.~Song.
\newblock High-order linearly implicit structure-preserving exponential
  integrators for the nonlinear {S}chr\"odinger equation.
\newblock {\em J. Sci. Comput.}, 90:66, 2022.

\bibitem{JQSZ2022-arXiv}
C.~Jiang, X.~Qian, S.~Song, and X.~Zheng.
\newblock Arbitrary high-order structure-preserving schemes for the generalized
  Rosenau-type equation.
\newblock {\em arXiv:2205.10241}, 2022.

\bibitem{JWG19}
C.~Jiang, Y.~Wang, and Y.~Gong.
\newblock Arbitrarily high-order energy-preserving schemes for the
  {C}amassa-{H}olm equation.
\newblock {\em Appl. Numer. Math.}, 151:85--97, 2020.

\bibitem{LWQ14}
H.~Li, Y.~Wang, and M.~Qin.
\newblock A sixth order averaged vector field method.
\newblock {\em J. Comput. Math.}, 34:479--498, 2016.

\bibitem{LW15}
Y.~Li and X.~Wu.
\newblock General local energy-preserving integrators for solving
  multi-symplectic {H}amiltonian {PDE}s.
\newblock {\em J. Comput. Phys.}, 301:141--166, 2015.

\bibitem{LWsina16}
Y.~Li and X.~Wu.
\newblock Functionally fitted energy-preserving methods for solving oscillatory
  nonlinear {H}amiltonian systems.
\newblock {\em SIAM J. Numer. Anal.}, 54:2036--2059, 2016.

\bibitem{Linares09}
F.~Linares and C.~Matheus.
\newblock Well-posedness for the 1{D} {Z}akharov--{R}ubenchik system.
\newblock {\em Adv. Differ. Eq.}, 14:261--288, 2009.

\bibitem{MHW2021jcp}
L.~Mei, L.~Huang, and X.~Wu.
\newblock Energy-preserving exponential integrators of arbitrarily high order
  for conservative or dissipative systems with highly oscillatory solutions.
\newblock {\em J. Comput. Phys.}, 442:110429, 2021.

\bibitem{MB16}
Y.~Miyatake and J.~C. Butcher.
\newblock A characterization of energy-preserving methods and the construction
  of parallel integrators for {H}amiltonian systems.
\newblock {\em SIAM J. Numer. Anal.}, 54:1993--2013, 2016.

\bibitem{Oliveira03pd}
F.~Oliveira.
\newblock Stability of the solitons for the one-dimensional
  {Z}akharov--{R}ubenchik equation.
\newblock {\em Phys. D.}, 175:220--240, 2003.

\bibitem{Oliveira08}
F.~Oliveira.
\newblock Adiabatic limit of the {Z}akharov--{R}ubenchik equation.
\newblock {\em Rep. Math. Phys.}, 61:13--27, 2008.

\bibitem{Oliveira15}
F.~Oliveira.
\newblock Stability of solutions of the {Z}akharov--{R}ubenchik equation.
\newblock {\em Wave Sta. Contin. Media}, pages 408--413, 2015.

\bibitem{Oruc21AMC}
\"O.~Oru\c{c}.
\newblock A radial basis function finite difference ({RBF}--{FD}) method for
  numerical simulation of interaction of high and low frequency waves:
  {Z}akharov--{R}ubenchik equations.
\newblock {\em Appl. Math. Comput.}, 394:125787, 2021.

\bibitem{Ponce05}
G.~Ponce and J.~C. Saut.
\newblock Well-posedness for the {B}enney-{R}oskes/{Z}akharov-{R}ubenchik
  system.
\newblock {\em Discret. Contin. Dyn. Syst.}, 13:818--852, 2005.

\bibitem{QM08}
G.~R.~W. Quispel and D.~I. McLaren.
\newblock A new class of energy-preserving numerical integration methods.
\newblock {\em J. Phys. A: Math. Theor.}, 41:045206, 2008.

\bibitem{Sanz88}
J.~M. Sanz-Serna.
\newblock Runge-{K}utta schemes for {H}amiltonian systems.
\newblock {\em BIT}, 28:877--883, 1988.

\bibitem{SV1986IMA}
J.~M. Sanz-Serna and J.~G. Verwer.
\newblock Conservative and nonconservative schemes for the solution of the
  nonlinear {S}chr\"odinger equation.
\newblock {\em IMA J. Numer. Anal.}, 6:25--42, 1986.

\bibitem{ST06}
J.~Shen and T.~Tang.
\newblock {\em Spectral and High-Order Methods with Applications}.
\newblock Science Press, Beijing, 2006.

\bibitem{SXY18}
J.~Shen, J.~Xu, and J.~Yang.
\newblock The scalar auxiliary variable {(SAV)} approach for gradient.
\newblock {\em J. Comput. Phys.}, 353:407--416, 2018.

\bibitem{SXY19siamrev}
J.~Shen, J.~Xu, and J.~Yang.
\newblock A new class of efficient and robust energy stable schemes for
  gradient flows.
\newblock {\em SIAM Rev.}, 61:474--506, 2019.

\bibitem{TS12}
W.~Tang and Y.~Sun.
\newblock Time finite element methods: a unified framework for numerical
  discretizations of {ODE}s.
\newblock {\em Appl. Math. Comput.}, 219:2158--2179, 2012.

\bibitem{Tapley-SISC2022}
B.~K. Tapley.
\newblock Geometric integration of {ODEs} using multiple quadratic auxiliary
  variables.
\newblock {\em SIAM J. Sci. Comput.}, 44:A2651--A2668, 2022.

\bibitem{WJJSC2022}
B.~Wang and Y.~Jiang.
\newblock Optimal convergence and long-time conservation of exponential
  integration for schr\"odinger equations in a normal or highly oscillatory
  regime.
\newblock {\em J. Sci. Comput.}, 90:1--31, 2022.

\bibitem{YZW17}
X.~Yang, J.~Zhao, and Q.~Wang.
\newblock Numerical approximations for the molecular beam epitaxial growth
  model based on the invariant energy quadratization method.
\newblock {\em J. Comput. Phys.}, 333:104--127, 2017.

\bibitem{ZR72}
V.~E. Zakharov and A.~M. Rubenchik.
\newblock Nonlinear interaction between high and low frequency waves.
\newblock {\em Prikl. Mat. Techn. Fiz.}, 5:84--89, 1972.

\bibitem{ZFPV95}
F.~Zhang, V.~M. P\'{e}rez-Garc\'{i}a, and L.~V\'{a}zquez.
\newblock Numerical simulation of nonlinear {S}chr\"{o}inger systems: a new
  conservative scheme.
\newblock {\em Appl. Math. Comput.}, 71:165--177, 1995.

\bibitem{ZJ2202Arix}
G.~Zhang and C.~Jiang.
\newblock Arbitrary high-order structure-preserving methods for the quantum
  {Z}akharov system.
\newblock {\em arXiv:2202.13052}, 2022.

\bibitem{Zhao14jsc}
X.~Zhao and Z.~Li.
\newblock Numerical methods and simulations for the dynamics of one-dimensional
  {Z}akharov--{R}ubenchik equations.
\newblock {\em J. Sci. Comput.}, 59:412--438, 2014.

\bibitem{WangZhang19acm}
X.~Zhou, T.~Wang, and L.~Zhang.
\newblock Two numerical methods for the {Z}akharov--{R}ubenchik equations.
\newblock {\em Adv. Comput. Math.}, 45:1163--1184, 2019.

\end{thebibliography}

\end{document}